\documentclass[12pt]{amsart}
\usepackage[top=70pt,bottom=70pt,left=75pt,right=77pt]{geometry}

\usepackage{amssymb, amsmath, amsthm, setspace, bbm, prettyref, graphicx, float, subfigure, color, mathrsfs, mathtools,comment}
\usepackage{soul}
\usepackage{Lutsko_Style}

\title{Effective Counting in Sphere Packings}
\author{Alex Kontorovich}
\thanks{Kontorovich is partially supported by NSF grant DMS-1802119, BSF grant 2020119, and the 2020-2021 Distinguished Visiting Professorship at the National Museum of Mathematics.}
\email{alex.kontorovich@rutgers.edu}
\address{Department of Mathematics, Rutgers University, New Brunswick, NJ}
\author{Christopher Lutsko}
\email{chris.lutsko@rutgers.edu}
\address{Department of Mathematics, Rutgers University, New Brunswick, NJ}

\begin{document}

  \maketitle
  \begin{abstract}
    \noindent
    
    Given a Zariski-dense, discrete group, $\Gamma$, of isometries acting on $(n+1)$-dimensional hyperbolic space, we use spectral methods to obtain a sharp asymptotic formula for the growth rate of certain $\Gamma$-orbits. In particular, this allows us to obtain a best-known effective error rate for the Apollonian and (more generally) Kleinian sphere packing counting problems, that is, counting the number of spheres in such with radius bounded by a growing parameter. 
    Our method extends the method of Kontorovich \cite{Kontorovich2009}, which was itself an extension of the orbit counting method of Lax-Phillips \cite{LaxPhillips1982}, in two ways. First, we remove a compactness condition on the discrete subgroups considered via a technical cut-off and smoothing operation. Second, we develop a coordinate system which naturally corresponds to the inversive geometry underlying the sphere counting problem, and give structure theorems on the arising Casimir operator and Haar measure in these coordinates.


            
  \end{abstract}


  \onehalfspacing
  \setlength{\abovedisplayskip}{1mm}

\section{Introduction}

The purpose of this paper is to give improved error estimates on the counting problem for Kleinian sphere packings (and discrete counting methods more broadly). A packing $\cP$ of $\mathbb{S}^n$ (thought of as the boundary of hyperbolic $(n+1)$-space, $\half^{n+1}$) is an infinite collection of round, disjoint balls whose union is dense in $\mathbb{S}^n$. Following \cite{KapovichKontorovich2021}, such is called 
{\it Kleinian} if its residual set (left over when the interiors of the balls are removed) agrees with the limit set of some discrete, geometrically finite subgroup, $\Gamma$, of isometries of $\half^{n+1}$. A familiar example 
is the classical Apollonian circle packing in $n=2$, see e.g \cite{Kontorovich2013} for more background and see Figure \ref{fig:ACP} for an example.

\begin{figure}[ht!]
  \begin{center}    
    \includegraphics[width=0.45\textwidth]{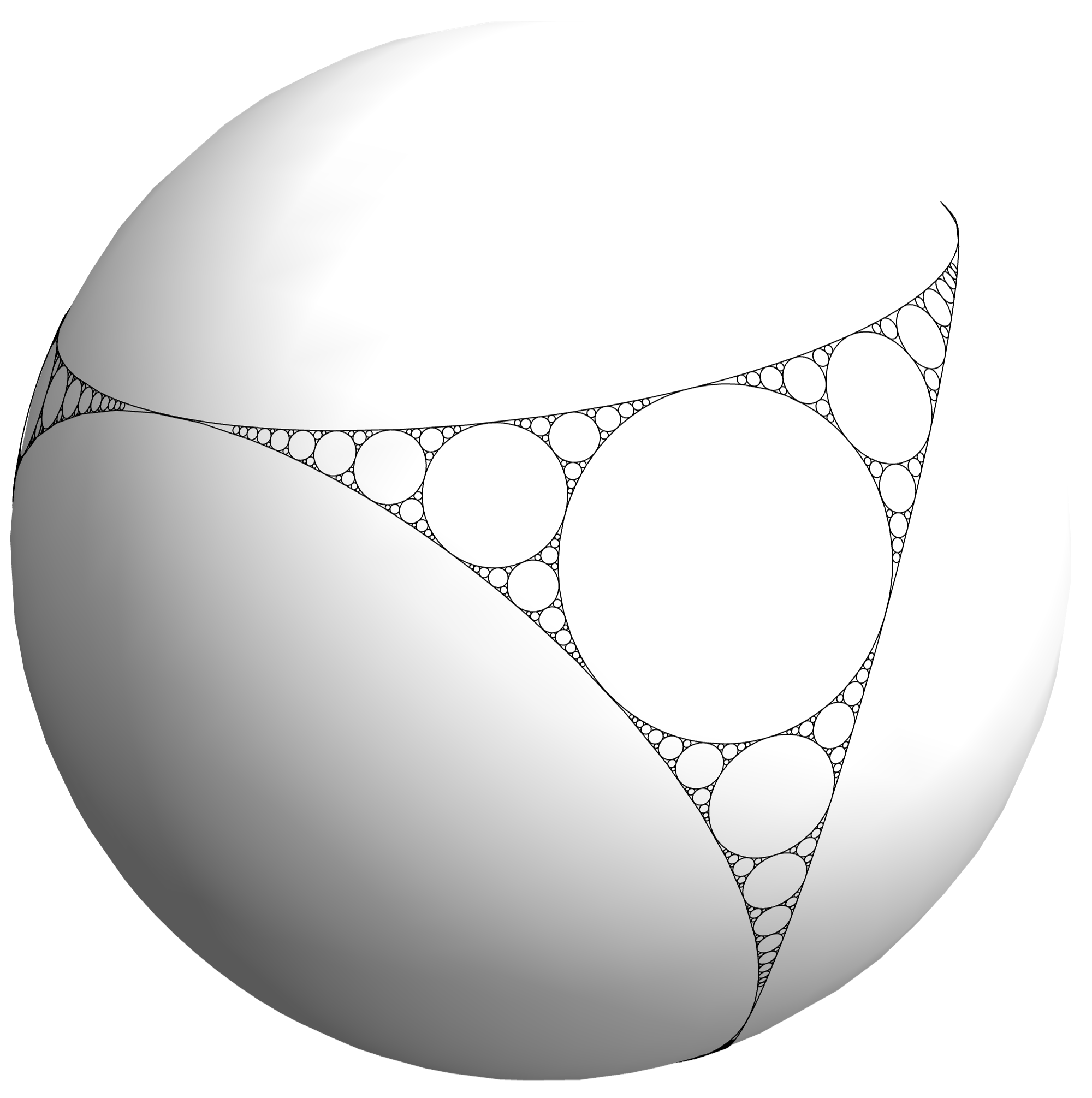}
  \end{center}
  \caption{%
    {\tt A classical Apollonian circle packing in the sphere $\mathbb{S}^{2}$. Note that the union of balls in the packing is dense in $\mathbb{S}^2$. (Image by Iv\'{a}n Rasskin.)
  }%
  }
  \label{fig:ACP}
\end{figure}

Let $\cP$ be a given Kleinian sphere packing. 
For a sphere $S\in\cP,$ let $b(S)$ denote its bend, that is, its (signed) inverse-radius; this is determined after a choice of coordinates on $\mathbb{S}^n$, and in particular a choice of a point at $\infty$ (see \S\ref{sec:ACP}). The Counting Problem is to estimate the number
$$
N_\cP (T) := \#\{S\in\cP:b(S)<T\}
$$
of spheres in $\cP$ with bend bounded by a parameter $T\to\infty.$ If the packing $\cP$ is bounded, that is, the chosen point at infinity is contained in the interior of some ball, then $N_\cP (T)$ is finite for all $T$; otherwise one can count spheres restricted to a bounded region (such as a period, if the packing is periodic).
Let $\delta = \operatorname{dim}(\cP)
$ be the Hausdorff dimension of the residual set of  $\cP$.

For the classical Apollonian packing, it is known that $\delta \approx 1.3$ (in general one has $n-1<\delta<n$). In this setting, Boyd \cite{Boyd1973} showed that $N_\cP(T) = T^{\delta + o(1)}$, which was improved in Kontorovich-Oh \cite{KontorovichOh2011} to an asymptotic formula, $N_\cP(T) \sim c_\cP T^\delta,$ where $c_\cP$ is a constant depending on the packing $\cP$. An effective power savings error rate was shown in \cite{Vinogradov2012} and \cite{LeeOh2013} independently. These tools and results have been generalized by many authors 
(see, e.g., \cite{Kim2015, MohammadiOh2015, Pan2017, EdwardsOh2021}). In this paper, we introduce a new method, modifying the approach in \cite{Kontorovich2009}, to produce a best known error exponent 
in the Counting Problem (including the classical 
Apollonian case).

The error exponent 
involves
 the spectrum of the hyperbolic Laplacian $\Delta
$ acting on $L^2(\G\bk \half^{n+1})$ where $\Gamma$ is the symmetry group of the packing.
From work of Lax-Phillips \cite{LaxPhillips1982}, Patterson \cite{Patterson1976}, and Sullivan \cite{Sullivan1984}, we have that the Laplace spectrum consists of a discrete isolated bottom eigenvalue $\gl_0=\gd(n-\gd)$, 
then possibly a finite number of further discrete eigenvalues in the ``exceptional'' range,
\be\label{eq:gls}
\gl_0<\gl_1\le\cdots\le\gl_k<n^2/4,
\ee
and purely continuous spectrum above $n^2/4$. (It was suggested by Sarnak \cite{SarnakToLagarias} that, in the case of the classical Apollonian packing, $k=0$, that is, there are no other discrete eigenvalues below $1$ except the bottom.\footnote{Added in print: Sarnak's question has now been answered in the affirmative \cite{KKL2022}.}) Write $\gl_j=s_j(n-s_j)$ with $n/2\le s_j<\gd$. 

These eigenvalues also correspond to the parameters of spherical complementary series representations occurring in the decomposition into irreducibles of $L^2(\Gamma\bk G)$, where $G=\SO(n+1,1)$, see \S\ref{sec:2p2}. We assume throughout that no nonspherical complementary series representations arise in this decomposition. This holds automatically for circle packings, that is, when $n=2$. Minor modifications are needed to handle the general case, leading to potentially worse error terms, see Remark \ref{rmk:X}.

  \begin{theorem} \label{thm:main}
    Given a Kleinian packing $\cP$, as above, there exists a constant $c_\cP>0$ such that:
    \begin{align} \label{Eqmain}
      N_{\cP}(T) = c_\cP T^{\delta}\left(1 + O\left(T^{-\eta} \right)\right)
    \end{align}
    as $T \to \infty$, where
    \be\label{eq:etaNew}
    \eta =\frac{2}{n+3}(\delta-s_1).
    \ee
    If there is no discrete spectrum other than the base eigenvalue, then we have that
     \begin{align*}
      N_{\cP}(T) = c_\cP T^{\delta}\left(1 + O\left( T^{-\eta}(\log T)^{2/(n+3)}\right)\right),
    \end{align*}
    where
    $
     \eta=\frac2{n+3} (\gd-n/2).
    $
  \end{theorem}

For the classical Apollonian packing $(n=2)$, our error exponent is $\eta = \frac2{5} (\gd-s_1)$, and if there are no discrete eigenvalues above the base then $\eta = \frac25(\gd-1)\approx0.12\dots$, whereas the best previously known exponent \cite[Theorem 1.1]{LeeOh2013} was:
\be\label{eq:etaOld}
\eta = \frac{2}{63}(\gd-s_1),
\ee
and if $s_1=n/2$, this exponent is $\eta=\frac2{63}(\gd-1)\approx 0.0097$. 
Hence \eqref{eq:etaNew} is a significant improvement over \eqref{eq:etaOld}.

\begin{remark}
  Counting with a smooth cutoff and extracting all of the lower order terms corresponding to eigenvalues other than the base, we obtain the best possible error exponent $\eta=\gd-n/2$ (see Theorem \ref{thm:w count}), which improves over the smooth error exponent in the Apollonian case $\eta=\frac27(\gd-1)$ in \cite[Theorem 8.2]{LeeOh2013}.
\end{remark}

\begin{remark}\label{rmk:X}
    If we remove the assumption that there are no nonspherical complementary series representations occurring in $L^2(\Gamma\bk G)$, then we must add an error term to \eqref{Eqmain} of order $T^{((n+1)\delta + 2(n-1))/(n+3)}$, that is, replacing $s_1$ in \eqref{eq:etaNew} with $n-1$; see Remark \ref{rmk:Y}. With more effort, this error term could be improved, see \cite{EdwardsOh2021}. 
\end{remark}

The proof of Theorem \ref{thm:main} introduces two new technical ideas: first, of independent interest, we overcome difficulties in the ``abstract spectral theory'' method of \cite{Kontorovich2009} arising from a non-compactness condition (see the discussion below), and second, we introduce a new decomposition tailored to sphere counting problems and derive structure theorems for the arising Casimir operator and Haar measure.

\subsection{Ideas in the Proof}\

\textbf{Previous Approach:} To explain the main ideas, we first recall the method introduced in \cite{Kontorovich2009}, which itself is modeled on \cite{LaxPhillips1982}. Consider the following related but simpler to exposit problem. Let $\G<\SL_2(\R)$ be a discrete, Zariski dense, geometrically finite subgroup, with critical exponent $\gd_\G>1/2$, acting on the upper half plane $\half^2,$ and assume that $\infty$ is not a \emph{point of approximation} for $\G$. This last condition implies that either $\infty$ is a cusp of $\G$ with stabilizer $\G_\infty$ (in which case the limit set is periodic), or $\infty$ is not in the limit set of $\G$ (and hence the limit set is bounded). Either way, consider the problem of counting
$$
N_\G(T) := \#\{\mattwos abcd \in\G_\infty\bk \G : c^2+d^2<T \}.
$$
(The Lax-Phillips work counts $a^2+b^2+c^2+d^2<T$, and new ideas are needed to handle the potential stabilizer when counting $c^2+d^2$.)
The assumption that $\infty$ is not a point of approximation assures that the set of such $c^2+d^2$ is discrete, and hence this count is finite for any $T$.
The main idea, as sketched below, is to use a particular function of the Laplacian to grow balls of size $T$ from balls of bounded size. 

In this subsection, write  $G=\SL_2(\R)$, $N=\{\mattwos 1\R01\},$ (so that $\G_\infty=\G\cap N$) and $K=\SO(2),$ and let
\be\label{eq:chiTdef0}
\chi_T :G\to\R:\mattwos abcd\mapsto \bo_{\{c^2+d^2<T\}}
\ee
be the indicator function of the region in question in $G$. Note that $\chi_T$ is left-$N$-invariant and right-$K$-invariant, and let
$$
F_T : \G\bk G / K\to \R : g \mapsto  \sum_{\g\in\G_\infty\bk \G}\chi_T(\g g),
$$
so that $N_\G(T)=F_T(e).$
As observed in \cite{Kontorovich2009}, $F_T$ is in $L^2(\G\bk G/K)=L^2(\G\bk\half)$ if and only if $\infty$ is a cusp of $\G$. 
To access the value of $F_T$ at the origin, we let $\psi$ be a smooth bump function about the origin in $G$, and automorphize it to $\Psi(g):=\sum_{\g\in\G}\psi(\g g)$. Then, we can write a smooth approximation of $N_{\Gamma}(T)$ as
\begin{align*}
    \wt{N_{\Gamma}}(T):= \<F_T,\Psi\>_{\Gamma}
\end{align*}
where the inner-product is with respect to $L^2(\Gamma \bk \half)$. Now suppose we take the inner-product of $F_T$ with an eigenfunction of the Laplacian $\phi$, with eigenvalue $\lambda = s(1-s)$. Then, solving a second order ODE, we would have that
\begin{align}\label{eq:FTinner}
    \<F_T,\phi\>_{\Gamma} = \alpha T^{s} + \beta T^{1-s},
\end{align}
for some $\alpha$ and $\beta$ depending on $\phi$.
The key idea of \cite{Kontorovich2009} is then to rewrite \eqref{eq:FTinner} in a way that involves the eigenvalue only, and not the coefficients $\alpha,\ \beta$ of the eigenfunction. This is achieved by setting $T=1$ and $T=b$ (for some $b<10$, see \S \ref{ss:Inserting SL2R}) in \eqref{eq:FTinner} and solving for $\alpha,\beta$, to give an expression of the form:
$$
\<F_T,\phi\> = K_T(s)\<F_1,\phi\>+L_T(s)\<F_b,\phi\>,
$$
for some functions $K_T,L_T$. This then allows one to prove the ``main identity'', which states that, in the sense of $L^2$, we have:
$$
F_T=K_T(\Delta)F_1+L_T(\Delta)F_b.
$$
This is exactly what we mean by ``growing'' the ball of radius $T$ from the Laplacian and bounded norm balls.
It is this identity that can be proved rigorously even in the absence of explicit Whittaker expansions and spectral decompositions.

\textbf{New Ideas:} Much of this approach fails if $\G$ does not have a cusp at $\infty$, and the main purpose of \cite{KontorovichOh2012} was to bypass this ``PDE'' approach and replace it with homogeneous dynamics, at the cost of worse error terms. The main new ideas of this paper allow us to recover the PDE approach (and then extend it to the Kleinian setting).

The first issue is that $F_T$ is not in $L^2$. The observation which eliminates this issue (made already in \cite{KontorovichOh2012}) is to add a second cut off in the $N$-direction without losing any of the orbit, since the limit set is anyway bounded; this adds compactness in the $x$-variable, restricted to some sufficiently large interval $[-X,X]$. Now we again unfold the inner product $\<F_T,\Psi\>_{\Gamma}$ leading to the following integral:
\begin{align*}
    \int_{0}^\infty \int_{-X}^{X} \chi_T(z) \wt{\chi}_X(z) \Psi(z) dx \frac{dy}{y^2}.
\end{align*}
We proceed with harmonic analysis on $[-X,X]\times (0,\infty)$. However this truncation requires a delicate smoothing procedure to avoid introducing boundary terms in the analysis of the Laplacian.
Once this is accomplished, the proof follows in a similar fashion. See \S\ref{ss:Inserting SL2R} for the details. 

In addition to overcoming the restriction in \cite{Kontorovich2009}, this $\SL_2(\R)$ result is of independent interest, and improves on \cite[Theorem 1.8]{KontorovichOh2012} which used methods from homogeneous dynamics to count Pythagorean triples.

{\bf In the Kleinian Setting:}
There are several further modifications and innovations needed to extend the above-described $\SL_2(\R)$ approach to the setting of orbits in circle/sphere packings. 
 In the previous setting, the stabilizer of $\chi_T$ in \eqref{eq:chiTdef0} was left-$N$ and right-$K$ invariant, and so it was natural to use  Iwasawa coordinates $\SL_2(\R) = NAK$. In the setting of sphere packings, one counts spheres in the orbit $S_0 \Gamma$ with bend less than $T$; here $S_0$ is a fixed $(n-1)$-sphere in $\widehat{\R^n}=\R^n\cup\{\infty\}=\dd \half^{n+1}$ and $\Gamma<G=\SO(n+1,1)$ is a symmetry group of the packing, acting on the right by M\"obius transformations.
 The analogous function $\chi_T$ is given by:
 $$
 \chi_T : G \to \R : g\mapsto \bo_{\{b(S_0 g)<T\}},
 $$
 where again $b(S)$ is the bend of a sphere $S$. This function is left-$H$ invariant, where $H=\Stab_G(S_0)\cong \SO(n,1)$; it is also right invariant under the group $L$ of affine motions, since translating a sphere does not change its bend. The latter decomposes further as $L=UM,$ where $U$ is a one-parameter unipotent group (which controls the co-bend, defined to be the bend of the inversion of a sphere through the unit sphere), and $M\cong \SO(n)$ rotates the sphere about the origin. It turns out that $H\cap M \cong \SO(n-1)$, and set $M_1 := M/(H\cap M)\cong \mathbb{S}^{n-1}$. The subgroup of $G$ which directly controls the bend is also a one-parameter unipotent group we call $\oU$, leading to the 
 map:
 $$
 H\times \oU \times U\times M_1 \to G,
 $$
which is an isomorphism in a neighborhood of the identity; see \S\ref{ss:decomp} for details.
An important feature of this decomposition is the fact that the Haar measure of $G$ in these coordinates decomposes as a product of $H$-Haar measure on the $H$ component, times the $M_1$-Haar measure on the $M_1$ component, and times something depending only on $U$ and $\oU$; see \S \ref{ss:Haar}. 
Moreover, in the proof, we only need the Casimir operator restricted to left-$H$- and right-$M_1$-invariant functions, for which we derive a nice, concise explicit expression in any dimension; see \S \ref{ss:Casimir}. 
 
 Let $\G_1:= \G\cap H$ denote the stabilizer of $S_0$ in $\G$. It follows from the Structure Theorem for Kleinian packings \cite[Theorem 22]{KapovichKontorovich2021} that $\G_1$ is a {\it lattice} in $H$ (that is, it acts with finite covolume); this fact will be used crucially in our analysis. 
 The finiteness of the volume of $\G_1\bk H$ is analogous 
 in the $\SL_2(\R)$ setting 
 to the finiteness of the volume of $K$, though the latter is trivial since $K$ is compact.
  Note that we acted on the left in $\SL_2(\R)$ and it is more convenient to act on the right for sphere packings. 
  
  In the $\SL_2(\R)$ setting, the $N$ direction was unbounded and required a cut-off. Analogously here, the $U$-direction is unbounded; this can be controlled via a similar truncation procedure by invoking the fact that the limit set is bounded in the $U$ direction, see Lemma \ref{lem:geometric}.

\begin{remark}
    Note that if the stabilizer of $\cP$ contains a full rank unipotent subgroup then the methods of \cite{Kontorovich2009} may be applied directly. For the existence of such, see \cite{KontorovichNakamura2019}.
\end{remark}

\subsection{Organization}

In section \ref{sec:prelim} we collect some preliminaries. In section \ref{sec:SL2R}, we warm up to the counting theorem and illustrate the main ideas by proving a result analogous to Theorem \ref{thm:main} in the $\SL_2(\R)$ setting.  In section \ref{sec:ACP}, we switch to the general $\SO(n+1,1)$ setting, and derive the Haar measure, and Casimir operator in the above-described coordinate system. Finally, in section \ref{s:Counting} we prove Theorem \ref{thm:main}.



\subsection{Acknowledgements}

The authors are grateful to Sam Edwards and Stephen D. Miller for enlightening conversations and suggestions, and to Iv\'{a}n Rasskin for allowing us to use his image in Figure \ref{fig:ACP} generated from his {\tt polytopack} Mathematica package (in preparation).
Thanks also to Hee Oh for comments on an earlier draft.

\section{Preliminaries}
\label{sec:prelim}

\subsection{Lie algebras and the Casimir Operator}  
\label{sec:2p1}\ 

We collect here some standard facts about the group $G=SO^\circ(n+1,1)$ (where $\circ$ denotes the connected component of the identity), and its Lie algebra $\fg=Lie(G)$ of dimension $d:=\dim(\fg)=(n+1)(n+2)/2$. 

In general,
Casimir operators 
generate
the center of the universal enveloping algebra $\mathcal U(\fg)$. 
In our rank-one setting, we can compute the Casimir operator $\cC$ as follows.
Let $X_1, \dots, X_d$ be a basis for the Lie algebra $\fg$, and let $X_1^\ast,\dots , X_d^\ast$ be a dual basis with respect to the 
Killing form: 
\begin{align*}
    B(X,Y) = \operatorname{Tr}(\ad(X) \circ \ad(Y)),
\end{align*}
that is,
$B(X_i,X_j^\ast)=\bo_{\{i=j\}}$. Then the Casimir operator can be expressed as
\begin{align*}
    \cC = \sum_{i=1}^d X_i X_i^\ast. 
\end{align*}
Since the elements of the Lie algebra act like first order differential operators, the Casimir operator acts as a second order differential operator on smooth functions on $G$; see \S\ref{ss:Casimir} for a detailed calculation in our setting. 

Let $K\cong SO(n+1)<G$ be a maximal compact subgroup.
When restricted to $K$-invariant smooth functions on $G$, 
the Casimir operator $\cC$ agrees (up to constant) with the hyperbolic Laplacian $\Delta$ under an identification $G/K\cong \half^{n+1}$.

\subsection{Decomposition of $L^2(\G \bk G)$ into irreducibles}
\label{sec:2p2}\ 

Let $\Gamma$ be a discrete,  geometrically finite, Zariski dense subgroup of $G$.
Given Iwasawa coordinates $G=NAK$, let $M\cong\SO(n)$ be the centralizer of $A$ in $K$.
The group $G$ acts  by the right-regular representation on the Hilbert space $\cH := L^2(\Gamma\bk G)$ of square-integrable $\Gamma$-automorphic functions. 
Recalling the standing assumption that $\cH$ does not weakly contain any nonspherical complementary series representations,
the space $\cH$ decomposes into components as follows:
\be\label{eq:cHdecomp}
\cH=\cH_0\oplus \cH_1\oplus\cdots\oplus\cH_k\oplus \cH^{tempered}. 
\ee
Here each $\cH_j$ is the $G$-span of the eigenfunction corresponding to the eigenvalue $\gl_j=s_j(n-s_j)$ in \eqref{eq:gls}, each of which is an irreducible spherical complementary series representation with parameter $s_j>n/2$,
and $\cH^{tempered}$ denotes tempered spectrum.
Moreover, the subspace $\cH_j^K$ of $K$-fixed vectors in $\cH_j$ is 1-dimensional, and spanned by the corresponding eigenfunction in $L^2(\G\bk\half^{n+1})\cong L^2(\G\bk G)^K$.
In general,  nonspherical complementary series representations can only occur if $n\ge 3$ and parameter $s\le n-1$ \cite{Knapp2001}.

\subsection{Abstract Spectral Theorem}
\label{sec:2p3}\ 

We recall the abstract spectral theorem (see for example \cite[Ch. 13]{Rudin1973}) for unbounded self-adjoint operators. Let $L$ be a self-adjoint, positive semidefinite operator on the Hilbert space $\cH$. In our applications $\cH$ will be either $L^2(\Gamma \bk G)$ or a subspace thereof and $L$ will be 
the Casimir operator.

\begin{theorem}[Abstract Spectral Theorem]\label{thm:AST}
    There exists a spectral measure $\nu$ on $\R$ and a unitary spectral operator $\wh{\phantom{\cdot\cdot}}: \cH \to L^2([0,\infty),d \nu)$ such that:
    \begin{enumerate}[label = \roman*)]
    \item Abstract Parseval's Identity: for $\phi_1,\phi_2 \in \cH$
        \begin{align}\label{API}
            \<\phi_1,\phi_2\>_{\cH}=\<\wh{\phi_1},\wh{\phi_2}\>_{L^2([0,\infty), d\nu)};
        \end{align}
    \item The spectral operator is diagonal with respect to $L$: for $\phi\in \cH$ and $\lambda \ge 0 $,
        \begin{align}
            \wh{L\phi}(\lambda) = \lambda \wh{\phi}(\lambda);
        \end{align}
    \end{enumerate}
\end{theorem}

Moreover, if $\lambda$ is in the point specturm of $L$ with associated 
eigenspace $\cH_\gl$, then for any $\psi_1,\psi_2\in \cH$ one has
        \begin{align}
            \wh{\psi_1}(\lambda)\wh{\bar\psi_2}(\lambda) 
            =
            \<\operatorname{Proj}_{\cH_\gl}\psi_1,
            \operatorname{Proj}_{\cH_\gl}\psi_2\>,
        \end{align}
where $\operatorname{Proj}$ refers to the projection to the subspace $\cH_\gl$. In the special case that $\cH_\gl$ is one-dimensional and spanned by the normalized eigenfunction $\phi_\gl,$  we have that
        \begin{align}
            \wh{\psi_1}(\lambda)\wh{\bar\psi_2}(\lambda) 
            =
            \<\psi_1,\phi_\gl\>
            \<\phi_\gl,\psi_2\>.
        \end{align}

\section{The $\SL_2(\R)$ Case}\label{sec:SL2R}

In this section, let $G:= \SL_2(\R)$ and $\Gamma < G$ be a Zariski dense, finitely generated, discrete subgroup with $\gd_\G>1/2$. 
The goal of this section is to prove Theorem \ref{thm:SL2R} below, that is, to improve on \cite[Theorem 1.11]{KontorovichOh2012} by extending the proof of \cite[Theorem 1.3 (1)]{Kontorovich2009} to the setting where $\Gamma_\infty$ is trivial. This will serve as a model for the method that we will generalize to higher dimensions in the rest of the paper.

Again, from work of Lax-Phillips \cite{LaxPhillips1982} and Patterson \cite{Patterson1976} we have that the Laplace spectrum below $1/4$ consists of a finite number of discrete eigenvalues 
\be\label{eq:SL2Rlambda}
\lambda_0 =\delta_\G(1-\gd_\G)< \lambda_1 \le \dots \le \lambda_k < 1/4.
\ee 
Write $\lambda_j = s_j(1-s_j)$ with $s_j\in(\frac12,1)$.
\begin{theorem}\label{thm:SL2R}
    Let $G:=\SL_2(\R)$ and $\Gamma < G$ a Zariski dense, finitely generated, discrete subgroup with $\gd_\G>1/2$. Assume that $\infty$ is either a cusp for $\G$ with stabilizer $\G_\infty$ or $\infty$ is not in the limit set of $\Gamma$. Then there exist constants $c_0>0$, $c_1, \dots, c_k$, and $\eta>0$ such that as $T \to \infty$
    \begin{align}\label{SL2R Count}
    \begin{aligned}
        N_\Gamma(T): &= \#\{\mattwos abcd \in\G_\infty\bk \G : c^2+d^2<T \}\\
                     &= c_0T^{\gd_{\G}} +c_1 T^{s_1} + \dots c_k T^{s_k}+ O(T^{\eta}\log^{1/2}T) ,
    \end{aligned}
    \end{align}
    with $\eta = \frac12(\delta_\G +\frac 12)$.
    
\end{theorem}

\begin{remark} The case where $\infty$ is a cusp is the content of \cite[Theorem 1.3 (1.5)]{Kontorovich2009}, so we assume below that $\infty$ is not in the limit set of $\Gamma$.
\end{remark}
\begin{remark}
  Note that in this case, since we are counting points in $\half$ (which is $K$ invariant)  we can extract all lower order terms corresponding to eigenvalues other than the base. 
\end{remark}
\begin{remark}
The corresponding error term in \cite[\S4.1]{KontorovichOh2012} is significantly worse (and not even explicitly specified) as compared to
Theorem \ref{thm:SL2R}, due in part to much worse dependence on Sobolev norms of the corresponding test vectors.
\end{remark}

To begin the proof, we proceed as described in the Introduction. For $g = \mattwos abcd\in G$, let
\begin{align}
    \chi_{T  }(g) := \begin{cases}
    1 & \mbox{ if } c^2+d^2 < T                                                      ,\\
    0 & \mbox{ otherwise.}
    \end{cases}
\end{align}
Under the identification $G/K\sim \half$ (where $K=SO(2)$), $g\mapsto z=g i$, $\chi_{T}$ corresponds to the conditions $\Im (z)>1/T$.

Now average $\chi_{T}$ over the group $\G$:

\begin{align}
    F_{T}(g):= \sum_{\gamma \in \Gamma} \chi_{T}(\gamma g).
\end{align}
such that the count in \eqref{SL2R Count} can be written as $N_\G(T)=F_T(e)$.

To access the value of $F_{T}$ at the identity, we follow the standard procedure of smoothing the count. To this end, 
we fix once and for all a smooth, even, nonnegative, compactly supported bump function $\psi_1\in C_0^\infty(\R)$ with unit total mass, $\int_\R\psi_1dx=1.$
Given $\vep>0$, we set $\psi_\vep(x):=\frac1\vep \psi_1(\frac x\vep).$
Write Iwasawa coordinates as $n_x:=\mattwos 1x01\in N,$ $a_y:=\diag(\sqrt y,1/\sqrt y)\in A$, and $k\in K=\SO(2)$, and by abuse of notation, write 
$\psi:G\to\R_+$ as follows:
\begin{align}
    \psi(n_x a_y k):=\psi_\vep(x)\psi_\vep(\log y).
\end{align}
Clearly $\psi$ is right-$K$-invariant, and it is easy to compute that $\int_G\psi(n_x a_yk)\frac{dx\, dy\, dk}{ y^2}=1+O(\vep)$.

We automorphise $\psi$ by setting:
$$
\Psi(g):=\sum_{\g\in\G}\psi(\g g),
$$
and consider the smoothed count:
$$
\widetilde N_\Gamma(T) := \< F_{T}, \Psi\>
.
$$
After unfolding $F_{T}$, we see that
\begin{align*}
\widetilde N_\Gamma(T) := 
\sum_{\g\in\Gamma} w_T(\g)
,
\end{align*}
where
 $w_T=w_{T,X,\vep}:\half \to [0,1]$ is given by
\begin{align}
    w_T(g):= 
    \int_{G} \chi_{T}(gh)\psi( h) dh
    .
\end{align}

The following theorem, from which Theorem \ref{thm:SL2R} follows by optimizing error terms, gives an asymptotic expansion for the smoothed count in the $\SL_2(\R)$-setting.

\begin{theorem}\label{thm:smooth SL2R}
    Assume that $\vep>0$ is small enough. Then there exist constants $c_{\Gamma,\vep}^{(i)}$ for $i = 0 ,1, \dots, k$ such that
    
    \begin{align}
        \wt{N}_\Gamma(T) = c_{\G,\vep}^{(0)}T^{\gd_{\G}}+ c_{\G,\vep}^{(1)}T^{s_1} + \dots + c_{\G,\vep}^{(k)} T^{s_k}+ O \left(\frac{1}{\vep} T^{1/2} \log T \right)
    \end{align}
    with $c_{\G,\vep}^{(0)}>0$, and the implied constant depending only on $\Gamma$. Moreover $c_{\Gamma,\vep}^{(i)}=c^{(i)}(1+O(\vep))$ for all $i =0,1, \dots, k$.

\end{theorem}

It remains to prove Theorem \ref{thm:smooth SL2R}.

\subsection{Inserting the Laplacian}
\label{ss:Inserting SL2R}
The smooth count above is an $L^2$ inner product of the indicator function $F_{T}$ with a smooth bump function $\Psi$. The key idea now is to forget about $\Psi$ and analyze the structure of the inner product of $F_{T}$ with {\it any} smooth $L^2$ function. 

Following \cite[\S 3]{Kontorovich2009} let
\begin{align}\label{KL def}
  K_T(s) : = \frac{T^{s}b^{1-s} - T^{1-s}b^s}{b^{1-s}-b^s}, \qquad
  L_T(s) : = \frac{ T^{1-s}-T^s}{b^{1-s}-b^s},
\end{align}
where  $b>1$ is a constant which can be taken to be less than $3$ (see \cite[(3.7)]{Kontorovich2009}). With that, if $\phi$ \emph{were} an eigenfunction of the Laplacian then we would have $\<F_T,\phi\> = A_{\phi}T^{s}+B_{\phi}T^{1-s}$. In this case, we could conclude that
\begin{align*}
  \<F_T,\phi\> = K_T(s)\<F_1,\phi\> + L_T(s) \<F_b,\phi\>.
\end{align*}
Thus we would like to show that 
\begin{align}\label{main ident}
    F_T= K_T(\gD)F_1 + L_T(\gD)F_b.
\end{align}
However, there is a problem created by the fact that $F_T$ is not in $L^2(\Gamma \bk \half)$.
To get around this, we will first restrict the support of $F_T$ in the $x$-direction before applying various smoothing arguments to conclude that a version of \eqref{main ident} holds for the modified $F_T$.
Since $\infty$ is not in the limit set, it is in the free boundary, and hence
 there exists a fixed an $X = X(\Gamma)>0$ such that the full region $((-\infty,-X]\cup[X,\infty))\times [0, \infty) \subset \half$ is contained in a single fundamental domain, see Figure \ref{fig:Xdomain}. 

\textbf{Restricting the Real Direction:} 
Define the following counting function 
\begin{align*}
    F_{T,X}(z):= \sum_{\gamma \in \Gamma} \chi_{T}(\gamma z) \wt{\chi}_X(\gamma z)
\end{align*}
where 
\begin{align*}
    \wt{\chi}_X(x+iy):= \begin{cases}
      1 &\mbox{ if } x \in [-X, X],\\
      0 &\mbox{ if } x \not\in [-X,X].
    \end{cases}
\end{align*}
Note that 
by our choice of $X$, we still have that $N_{\Gamma}(T)= F_T(i) = F_{T,X}(i)$.
\begin{figure}[ht!]
  \begin{center}    
    \includegraphics[width=0.6\textwidth]{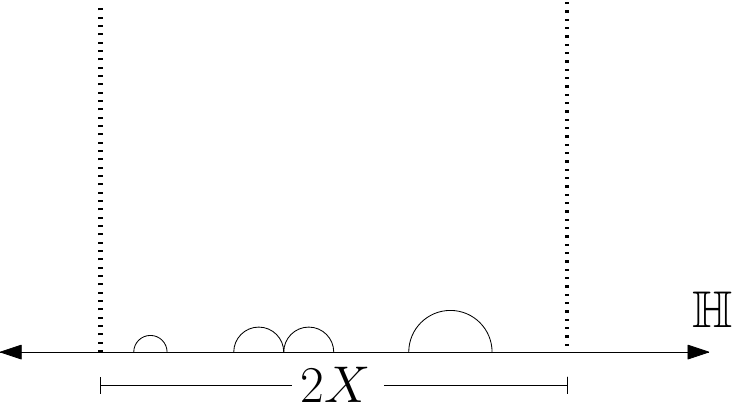}
  \end{center}
  \caption{%
    {\tt Above we sketch a fundamental domain of $\Gamma$ which extends to infinity in the real direction, and the cut-off due to $\wt{\chi}_X$ which is large enough such that $[-X,X]$ contains the entire limit set.
  }%
  }
  \label{fig:Xdomain}
\end{figure}

Now, consider the difference operator
$$
    G_{T,X}:=F_{T,X}-K_T(\gD)F_{1,X} - L_T(\gD)F_{b,X}.
$$
By self-adjointness of $\gD$, for any $\Psi \in L^2(\Gamma \bk \half)$ we have
\begin{align*}
    &\< G_{T,X},\Psi\>_{\G\bk\bH}= \\
    &\int_{\G\bk\bH}
    \left[
    F_{T,X}(z)\Psi(z) - F_{1,X}(z)(K_T(\gD)\Psi)(z) - F_{b,X}(z)(L_T(\gD)\Psi)(z)
    \right]dz,
\end{align*}
which we can unfold to
\begin{align}\label{eq:100501}
\begin{aligned}
    &\< G_{T,X},\Psi\>_{\G\bk\bH}=
  \int_\bH (\chi_T(z)\wt{\chi}_{X}(z)\Psi(z) - \\
  &\phantom{+++++}\chi_1(z)\wt{\chi}_X(z)(K_T(\gD)\Psi)(z) - \chi_b(z)\wt{\chi}_X(z)(L_T(\gD)\Psi)(z))dz.
    \end{aligned}
\end{align}
 Our goal is to show, that for any $\Psi\in L^2(\Gamma \bk \half)$ we have $\<G_{T,X},\Psi\>_{\Gamma} = 0$. This implies the following identity

\begin{proposition}\label{prop:main ident n}
  For any values of $T$ and $X$ large enough, we have
  \begin{align}
      F_{T,X} = K_T(\Delta) F_{1,X} + L_T(\Delta) F_{b,X}.
  \end{align}
  where $K_T$ and $L_T$ are the differential operators defined above and $b$ fixed. 
\end{proposition}

\textbf{Smoothing $\chi_T$:} Now we smooth the function $\chi_T$ so that we can move the Laplacian over to $\wt{\chi}_X$ via self-adjointness.  Let $\sigma>0$, and define the following smooth cut-off function
\begin{align*}
  \chi_{1,\sigma}(x+iy) := \begin{cases}
    1 & \mbox{ if } y > 1,\\
    0 & \mbox{ if } y < (1-\sigma),
  \end{cases}
\end{align*}
and smooth, in between the two bounds. Now let $\chi_{T,\sigma}(x+iy) := \chi_{1,\sigma}(x+i T y).$

Let $G_{T,X}^{\sigma}$ be defined similarly to $G_{T,X}$, with $\chi_T$ replaced by $\chi_{T,\sigma}$. That is, let 
$$
\cI_X := [-X,X]
$$ 
and let
\begin{align*}
  &G_{T,X}^{\sigma}(\Psi) : =
  \int_0^\infty \int_{x\in\cI_X} \wt{\chi}_{X}(z)\\
  &\phantom{+}\left(\chi_{T,\sigma}(z) \Psi(z) - \chi_{1,\sigma}(z) (K_{T}(\gD)\Psi)(z) - \chi_{b,\sigma}(z) (L_{T}(\gD)\Psi)(z) \right) dx \frac{dy}{ y^2}.
\end{align*}
Note that by construction for any fixed $t$
\begin{align*}
  \int_0^\infty \int_{x\in\cI_X} \abs{\chi_{t,\sigma}(z) - \chi_{t}}^2 dx \frac{dy}{y^2} \ll_{t,X} \sigma.
\end{align*}
Thus, using Cauchy-Schwarz we have that for any $\Psi \in L^2(\Gamma \bk \half)$ we have that
\begin{align*}
  \lim_{\sigma \to 0} G_{T,X}^{\sigma}(\Psi) = \<G_{T,X},\Psi\>_\Gamma. 
\end{align*}
 Now our goal, is to establish the following lemma, note that we have already fixed $T$, $X$ large enough, $\sigma>0$ small enough.

\begin{lemma}\label{lem:G zero}
 
   For any function $\Psi\in L^2(\Gamma\bk \half)$ we have that $G_{T,X}^{\sigma}(\Psi) =0 $.  
\end{lemma}
 To prove this lemma, fix $\epsilon>0$ small enough, we will show that $\abs{G_{T,X}^{\sigma}(\Psi)} < \epsilon$. 

\vspace{4mm}

\textbf{Periodizing and Smoothing $\Psi$:} Now we periodize and smooth $\Psi$ in order to do spectral theory on $\cI_X \times (0,\infty)$ which we call the space $\cF_X$. Note that $\cF_X \cong \Xi_\infty \bk\half $ for an elementary group $\Xi_\infty$, but this coordinate description is more useful when we generalize to higher dimensions. Let $L^2(\cF_X)$ denote the $L^2$ space with measure given by $\frac{dx dy}{y^2}$. Let $\widetilde\Psi(z)$ denote a function which agrees with $\Psi$ on 
$$
[-X,X-\eta)\times(0,\infty)\subset\bH,
$$
where $\eta>0$ is to be chosen later. From $x= X-\eta$ to $x=X-\eta/2$, $\widetilde\Psi(z)$ smoothly interpolates between the values of $\Psi(z)$ and $\Psi(z-2X)$, and from $x=X-\eta/2$ up to $x=X$, $\widetilde \Psi$ is exactly equal to $\Psi(z-2X)$, ensuring that all derivatives of $\widetilde\Psi$ at the boundary values $x=X$ and $x=-X$ agree.


Note that the $L^2(\cI_X \times [1/2T,\infty))$ cost of moving from $\Psi$ to $\wt\Psi$ is small. Indeed,
\begin{align*}
  \operatorname{diff}(\Psi)
  &:=
  \int_{1/2T}^\infty\int_{X-\eta}^{X}
  \left| \widetilde\Psi(z)- \Psi(z)\right|^2
  dx{dy\over y^2}\\
  &\ll
  \int_{1/2T}^\infty \int_{X-\eta}^{X}
  |\Psi(z-2X)|^2+ |\Psi(z)|^2
  dx{dy\over y^2}.
\end{align*}
Since $\Psi\in L^2(\G\bk\bH)$, as $\eta\to0$, this integral goes to zero, and hence by choosing $\eta$ small enough, we can make the difference less than $\epsilon^2$.

Now note that
\begin{align*}
     &\int_0^\infty\int_{\cI_X} \wt{\chi}_X(z)(\chi_{T,\sigma}(z)-K_{T}(\Delta)\chi_{1,\sigma}(z)-L_{T,\sigma}(\Delta)\chi_{b,\sigma}(z))  \Psi(z)dx\frac{dy}{y^2}\\
  &=
     \int_{1/2T}^\infty\int_{\cI_X} \wt{\chi}_X(z)(\chi_{T,\sigma}(z)-K_{T}(\Delta)\chi_{1,\sigma}(z)-L_{T}(\Delta)\chi_{b,\sigma}(z)) (\Psi(z)-\wt\Psi(z)) dx\frac{dy}{y^2}\\
     &+ \int_{1/2T}^\infty \int_{\cI_X}\wt{\chi}_X(z)(\chi_{T,\sigma}(z)-K_{T}(\Delta)\chi_{1,\sigma}(z)-L_{T}(\Delta)\chi_{b,\sigma}(z)) \wt\Psi(z) dx\frac{dy}{y^2},
\end{align*}
and by Cauchy Schwarz and the above argument
\begin{align}
  \notag
  &\int_{1/2T}^\infty\int_{\cI_X}\wt{\chi}_X(z) (\chi_{T,\sigma}(z)-K_{T}(\Delta)\chi_{1,\sigma}(z)-L_{T}(\Delta)\chi_{b,\sigma}(z)) (\Psi(z)-\wt\Psi(z)) dx\frac{dy}{y^2}\\
  &\ll \left(\operatorname{diff}(\Psi)\right)^{1/2}\left(\int_0^\infty\int_{\cI_X} |(\chi_{T,\sigma}(z)-K_{T}(\Delta)\chi_{1,\sigma}(z)-L_{T}(\Delta)\chi_{b,\sigma}(z))|^2 dx\frac{dy}{y^2}\right)^{1/2}\notag\\
  \label{Psi sharp bound}
  &\phantom{+++++++++} \ll_{\sigma,T, X} \epsilon.
\end{align}

Hence it remains to show that
\begin{align} \label{epsilon goal}
  \int_0^\infty\int_{\cI_X} \wt{\chi}_X(z)(\chi_{T,\sigma}(z)-K_{T}(\Delta)\chi_{1,\sigma}(z)-L_{T}(\Delta)\chi_{b,\sigma}(z)) \wt\Psi(z) dx\frac{dy}{y^2} = 0.
\end{align}

\vspace{4mm}

\textbf{Working on $\cF_X$:} To prove \eqref{epsilon goal} let 
$$g_{T,X}^{\sigma}(z): =\wt{\chi}_X(z)\left(\chi_{T,\sigma}(z)-K_{T}(\Delta)\chi_{1,\sigma}(z)-L_{T}(\Delta)\chi_{b,\sigma}(z)\right).  $$
Thus, we are left with 
\begin{align*}
  \< g^{\sigma}_{T,X},  \wt\Psi  \>_{\cF_X}.
\end{align*}
Lemma \ref{lem:G zero} will then follow from the following lemma, which shows that, for an arbitrary $\psi \in L^2(\cF_X)$ the inner product $\<g_{T,X}^{\sigma},\psi\>$ is not correlated with any almost eigenfunction:
\begin{lemma}\label{lem:gn ineq}
   Fix $T$, $X$, and $\sigma$ as above, then for any $\psi \in L^2(\cF_X)$ and any $\lambda \ge 0$, we have
  \begin{align}
    \<g_{T,X}^{\sigma},\psi\>_{\cF_X} \ll_{\lambda,T,\sigma,X} \|(\Delta-\lambda)\psi\|_{\cF_X}.
  \end{align}
\end{lemma}

  \begin{proof}[Proof of Lemma \ref{lem:gn ineq}]
  Fix $\psi \in L^2(\cF_X)$ and consider
\begin{align*}
  \< \chi_{T,\sigma}\wt{\chi}_X(z),  \psi  \>_{\cF_X}  = \int_0^\infty \chi_{T,\sigma}(y) \int_{\cI_X} \wt{\chi}_X(x) \psi(z) dx \frac{dy}{y^2},
\end{align*}
and define $f(y): = \int_{\cI_X}  \wt{\chi}_X(x)\psi(z) dx$ and $h(y): = \int_{\cI_X}(\Delta - \lambda) \wt{\chi}_X(x)\psi(z) dx$.  Then we note that, by periodicity in the $x$-direction,
\begin{align*}
  h(y) &=  \int_{\cI_X} (\Delta - \lambda)\wt{\chi}_X(x) \psi(z) dx\\
       &= - y^2 \partial_{yy}f(y) - \lambda f(y).
\end{align*}
Thus, \cite[Lemma B.1]{Kontorovich2009} (which is a simple application of the method of variation of parameters) implies
\begin{align}\label{diff eq}
  f(y) = \alpha y^s + \beta y^{1-s}+ y^{s} u(y) + y^{1-s}v(y),
\end{align}
where
\begin{align*}
  u(y) := (1-2s)^{-1}\int_{(1-\sigma)/T}^y w^{-1-s}h(w) dw,
\end{align*}
\begin{align*}
  \text{ and } \
  v(y):= (1-2s)^{-1}\int_{(1-\sigma)/T}^y w^{s-2}h(w) dw,
\end{align*}
if $\lambda \neq 1/4$. And 
\begin{align}\label{diff eq 2}
  f(y) = \alpha y^{1/2} + \beta y^{1/2}\log y +  y^{1/2} u(y) + v(y) y^{1/2} \log y,
\end{align}
where
\begin{align*}
  u(y) := \int_{(1-\sigma)/T}^y w^{-3/2}\log(w)h(w) dw,
\end{align*}
\begin{align*}
  \text{and} \qquad
  v(y):= -\int_{(1-\sigma)/T}^y w^{-3/2}h(w) dw.
\end{align*}
Therefore (assuming $\lambda \neq 1/4$ for simplicity) integrating the $y$ variable gives
\begin{align*}
  \< \chi_{T,\sigma}\wt{\chi}_X,\psi\>_{\Xi_\infty} = \int_{0}^\infty \chi_{T,\sigma}(y) (\alpha y^s + \beta y^{1-s})  \frac{dy}{y^2} + I + II
\end{align*}
where $I := \int_{0}^\infty \chi_{T,\sigma}(y)y^{s-2} u(y) dy$ and $II :=  \int_{0}^\infty \chi_{T,\sigma}(y) y^{-1-s} v(y) dy$. Now we can use Cauchy-Schwarz (as in \cite[(B.5)]{Kontorovich2009}) to establish that $I,II \ll \|(\Delta - \lambda) \psi \|_{\cF_X}$. In fact, this is the crucial reason why we needed to work on $\cF_X$ which is compact in the $x$ direction, and thus $\chi_{T,\sigma}$ is (square) integrable. Thus, we may conclude
\begin{align*}
  \< \chi_{T,\sigma},\psi\>_{\cF_X} = A_\sigma T^s +  B_\sigma T^{1-s} + O(\|(\Delta - \lambda) \psi \|_{\cF_X})
\end{align*}
where $A_\sigma = \beta \int_{0}^\infty\chi_{1,\sigma}(y) y^{-1-s} dy$ and $B_\sigma : = \alpha\int_{0}^\infty\chi_{1,\sigma}(y) y^{s-2} dy$.

Given our choice of $K_T$ and $L_T$ from \eqref{KL def} we have that
\begin{align}\label{K and L bounds}
    K_T(s),L_T(s)\ll 
    \begin{cases} 
        T^s & \mbox{ if } s \in (1/2,1],\\
        T^{1/2}\log T & \mbox{ if } s =1/2+it.
    \end{cases}
\end{align}
Then, by the analysis in \cite[Proposition 3.5]{Kontorovich2009} we have
\begin{align}\label{gn bound}
  \< g^{\sigma}_{T,X},  \psi  \>_{\cF_X} \ll_{T,\lambda,X,\sigma} \|(\Delta - \lambda) \psi \|_{\cF_X}
\end{align}
for any choice of $\psi \in L^2(\cF_X)$.
\end{proof}

From there we can choose $\psi$ to be  as in \cite[Proof of Theorem 3.2]{Kontorovich2009} to establish that $g^{\sigma}_{T,X}$ is almost everywhere $0$. Note that for this argument to work one only needs the bound \eqref{gn bound}, a Hilbert space (here $L^2(\cF_X)$), an unbounded self-adjoint operator (i.e $\Delta$), and the abstract spectral theorem.

From there we conclude that
\begin{align*}
  \< g_{T,X}^{\sigma},\wt\Psi\>_{\cF_X} =0
\end{align*}
From there, we conclude \eqref{epsilon goal}, namely
\begin{align} 
  \int_0^\infty\int_{\cI_X} \wt{\chi}_X(z)(\chi_{T,\sigma}(z)-K_{T}(\Delta)\chi_{1,\sigma}(z)-L_{T}(\Delta)\chi_{b,\sigma}(z)) \wt\Psi(z) dx\frac{dy}{y^2} = 0.
\end{align}
Then thanks to \eqref{Psi sharp bound} we conclude
\begin{align*}
  G_{T,X}^{\sigma}(\Psi) &= \int_0^\infty\int_{\cI_X}\wt{\chi}_X(z) (\chi_{T,\sigma}(z)-K_{T}(\Delta)\chi_{1,\sigma}(z)-L_{T,\sigma}(\Delta)\chi_{b,\sigma}(z))  \Psi(z)dx\frac{dy}{y^2}\\
  &\ll_{\sigma,T,X} \epsilon
\end{align*}
for any value of $\epsilon>0$. Taking $\epsilon$ to $0$ establishes  Lemma \ref{lem:G zero}. Since we have that for any $\Psi \in L^2(\Gamma \bk \half)$
\begin{align*}
  \lim_{\sigma \to 0} G_{T,X}^{\sigma}(\Psi)  = \<G_{T,X},\Psi\>_\Gamma
\end{align*}
we conclude that $\<G_{T,X},\Psi\>_\Gamma=0$. Which is exactly Proposition \ref{prop:main ident n}.
\qed

\subsection{Proof of Theorem \ref{thm:smooth SL2R}}

For the proof of Theorem \ref{thm:smooth SL2R}, we return to the smoothed count
\begin{align*}
    \wt{N}_{\Gamma}(T) = \<F_{T,X}, \Psi \>_{\Gamma}.
\end{align*}
Now apply the abstract Parseval's identity \eqref{API}
\begin{align}
    \<F_{T,X},\Psi\>_\Gamma 
    &= \<\wh{F_{T,X}},\wh{\Psi}\>_{\Spec(\Gamma)} \notag\\
    &= \wh{F_{T,X}}(\lambda_0)\wh{\Psi}(\lambda_0) + \int_{\Spec(\Gamma) \setminus \{\lambda_0\}} \wh{F_{T,X}}(\lambda)\wh{\Psi}(\lambda) \mathrm{d}\nu(\lambda),
  \label{spectral ip}
\end{align}
where for ease of exposition, we assume that $\gl_0$ is the only eigenvalue below $1/4$; in general, the other eigenvalues are dealt with similarly.


 Addressing the first term in \eqref{spectral ip}, we can use the abstract spectral theorem, and the fact that $F_{T,X}$ and $\Psi$ are $K$-fixed, to say
\begin{align*}
    \wh{F_{T,X}}(\lambda_0)\wh{\Psi}(\lambda_0) = \<F_{T,X},\phi_0\>\<\Psi,\phi_0\>,
\end{align*}
where $\phi_0$ is the base eigenfunction. 
Note that, for the first factor we can apply the main identity Proposition \ref{prop:main ident n}, and the definition to conclude
\begin{align*}
    \wh{F_{T,X}}(\lambda_0)\wh{\Psi}(\lambda_0) = T^\delta c \<\phi_0,\Psi\>_\Gamma + O(T^{1/2})
\end{align*}
for some constant $c$ independent of $T$. As for the second factor, by the mean value theorem (see \cite[(4.17)]{Kontorovich2009} for details) we have
\begin{align*}
  \<\phi_0,\Psi\>_\Gamma = \phi_0(i)+ O(\vep).
\end{align*}

It remains to bound the contribution to \eqref{spectral ip} from the remainder of the spectrum (assuming here that there are no isolated eigenvalues apart from the base). Using Proposition \ref{prop:main ident n} we have that
\begin{align*}
  \Err &:= \int_{\Spec(\Gamma)\setminus\{\lambda_0\}} \wh{F}_{T,X}(\lambda)\wh{\Psi}(\lambda)d\nu\\
  &= \int_{\Spec(\Gamma\setminus\{\lambda_0\}} \left(\wh{K_T(\Delta)F_{1,X}}(\lambda) + \wh{L_T(\Delta)F_{b,X}}(\lambda)
  \right)\wh{\Psi}(\lambda) d\nu.
\end{align*}
By the abstract spectral theorem and \eqref{K and L bounds} we have $\wh{K_T(\Delta)F_{1,X}}(\lambda)  \ll T^{1/2}\log T \wh{F_{1,X}}(\lambda)$. Thus, by Cauchy-Schwarz, positivity, and Parseval's identity
\begin{align*}
   &\int_{\Spec(\Gamma)\setminus\{\lambda_0\}} K_T(\lambda)\wh{F_{1,X}}(\lambda)\wh{\Psi}(\lambda)\mathrm{d}\nu
  \ll T^{1/2} \log T\int_{\Spec(\Gamma)\setminus\{\lambda_0\}}\wh{F_{1,X}}(\lambda)\wh{\Psi}(\lambda)\mathrm{d}\nu\\
    &\ll T^{1/2} \log T\left(\int_{\Spec(\Gamma)\setminus\{\lambda_0\}}\abs{\wh{F_{1,X}}(\lambda)}^2\mathrm{d}\nu\right)^{1/2}\left(\int_{\Spec(\Gamma)\setminus\{\lambda_0\}}\abs{\wh{\Psi}(\lambda)}^2\mathrm{d}\nu\right)^{1/2}\\
  &\ll T^{1/2}\log T
  \|F_{1,X} \|_\Gamma \|\Psi \|_\Gamma,
\end{align*}
Since $\Psi$ is an $\varepsilon$-approximation to the identity, and since the term involving $L_T$ can be treated similarly, we thus conclude
\begin{align*}
  \Err \ll \frac{1}{\vep} T^{1/2} \log T.
\end{align*}
This completes the proof of Theorem \ref{thm:smooth SL2R}.

\section{Kleinian Sphere Packings}\label{sec:ACP}

Turning now to the  
sphere packing setting, let $\cP$ be a fixed bounded Kleinian sphere packing. Given a sphere $S \in \cP$ let $b(S)$ denote the bend of $S$. Then our aim is to establish the asymptotic, \eqref{eq:etaOld} for
\begin{align*}
    N_{\cP}(T):= \#\{S \in \cP : b(S) < T\}.
\end{align*}

\subsection{Preliminaries on Inverse Coordinate Systems and the Symmetry Group}

We now give a model of hyperbolic space, and develop an inversive coordinate system for the spheres on the ideal boundary of such; see, e.g., \cite[\S 3.1]{KapovichKontorovich2021} for background.
To begin, we fix a real quadratic form $Q$ of signature $(n+1,1)$. 
For concreteness, we can change variables over $\R$ to the ``standard'' example of $Q=-x_0 x_{n+1}+x_1^2+\cdots +x_n^2$ which has half-Hessian
\begin{align}\label{eq:Qeg}
  Q = \begin{pmatrix}
    0 & 0 & -\frac{1}{2}  \\
    0 & I_n & 0 \\
    -\frac{1}{2} & 0 & 0 \\
  \end{pmatrix}.
\end{align}
Then the quadratic space $(V=\R^{n+1,1},Q)$ with product $v\star w = v\cdot Q\cdot w^t$ contains the cone $V_0:=\{v : Q(v)=0\}$, the one-sheeted hyperboloid $V_1=\{v: Q(v)=1\}$, and the two-sheeted hyperboloid $V_{-1}:=\{v:Q(v)=-1\}$; fix either component of the latter for our model of $\half^{n+1}$. The group $O_Q(\R)$ acts on $V_{-1}$, and its subgroup $G=O_Q^\circ (\R)$, that is, the connected component of the identity, fixes the components. 

\begin{figure}[ht!]
  \begin{center}    
    \includegraphics[width=0.5\textwidth]{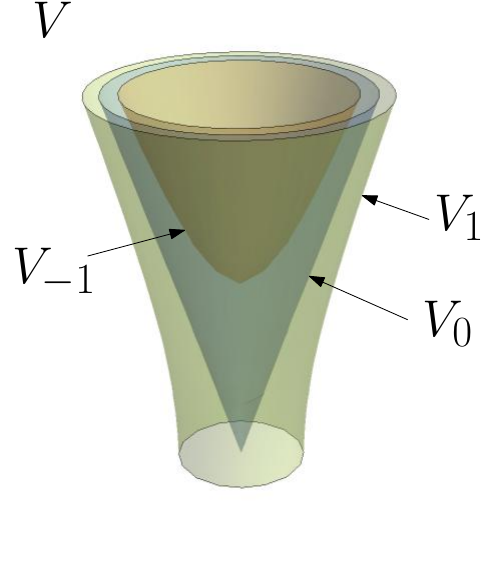}
  \end{center}
  \caption{%
    {\tt  In the above figure, we display (half of) the quadratic space $V$, and its components $V_0$, $V_1$ and $V_{-1}$. }
  }%
  
  \label{fig:components}
\end{figure}

There is a 1-1 correspondence between vectors $v\in V_1$ and (oriented) spheres in $\dd \half^{n+1}$, obtained as follows. Given such a $v$, the orthogonal space $v^\perp:=\{w\in V: v\star w =0\}$ intersects the fixed component of $V_{-1}$ at a hyperplane $\cong \half^n$, and the ideal boundary of the latter is the desired sphere.

\begin{figure}[ht!]
  \begin{center}    
    \includegraphics[width=0.7\textwidth]{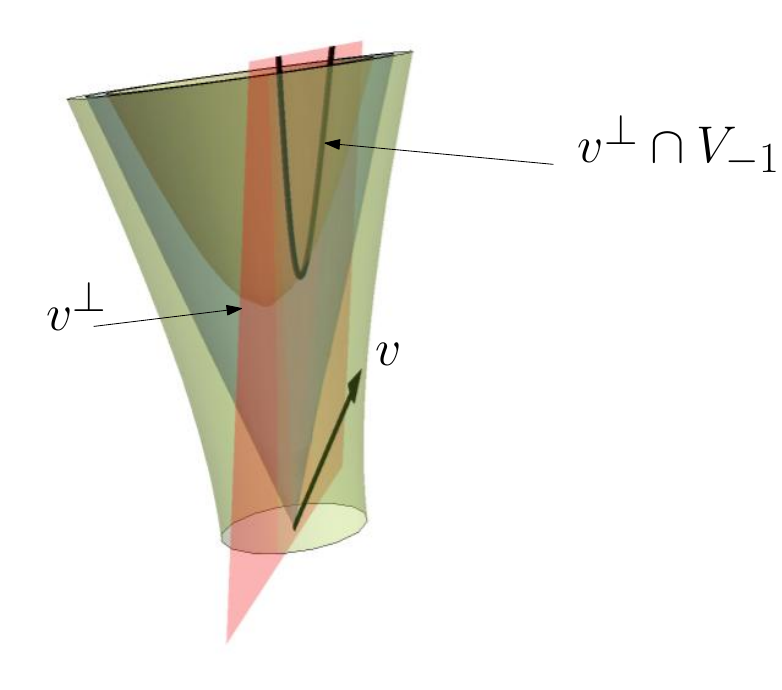}
  \end{center}
  \caption{%
    {\tt  Here we display a vector $v \in V_1$, its orthogonal space $v^\perp$ and the intersection this orthogonal space with $V_{-1}$ which corresponds to an ideal sphere.}
  }%
  
  \label{fig:sphere}
\end{figure}

This geometric correspondence is made algebraic after a choice of (inversive) coordinates on $V_1$ as follows. 
Let $V^*:=\{v^*:V\to\R,\text{ linear}\}$ be the dual space to $V$, and $Q^*$ the dual form, so that $v^*\star w^*=v\star w$.
Fix a non-zero null covector $b^*\in V^*$, that is, $Q^*(b^*)=0$. Also fix a null covector $\hat b^*$, with $b^*\star \hat b^*=-2$.
For the case of the standard form in \eqref{eq:Qeg}, one can make the choice $b^*=(0,\dots,0,-2)$ and $\hat b^*=(-2,0,\dots,0)$. Then one picks an orthonormal system, $bx_1^*,\dots,bx_n^*,$ for the orthogonal complement to the span of $b^*$ and $\hat b^*$. Then the sphere corresponding to the vector $v\in V_1$ has bend (that is, inversive radius)
\be\label{eq:bendIs}
b^*(v),
\ee
and center $\frac1{b^*(v)}(bx_1^*(v),\dots,bx_n^*(v))$. When $b^*(v)=0$, the sphere is a plane (which has no ``center''), so the expression $(bx_1^*(v),\dots,bx_n^*(v))$ is a unit normal to the plane. The ``co-bend'' of a sphere is defined as the bend of the image of the sphere on inversion through the unit sphere at the origin. The co-bend of the sphere corresponding to $v$ is given by $\hat b ^*(v)$. Therefore the tuple $(b^*(v),bx_1^*(v),\dots,bx_n^*(v),\hat b^*(v))$ gives a complete inversive coordinate system on $V_1$.
It is sometimes convenient to isolate the `bend-center', comprising the coordinates not including the first and last; so we define:
\be\label{eq:bStarZ}
bz^*(v):=(bx_1^*(v),\dots,bx_n^*(v)).
\ee

A Kleinian sphere packing decomposes into finitely many $\G$-orbits (by the Structure Theorem \cite[Theorem 22]{KapovichKontorovich2021}), which in the above coordinate system corresponds simply to orbits $v_0\cdot \G$, with $v_0\in V_1$, and the bends of such are measured by $b^*(v),$ for $v\in v_0\cdot \G$. We count the whole packing by counting one orbit at a time.

\subsection{Decomposition of $G$}
\label{ss:decomp}
Fix $v_0$ as above.
Let 
$\chi_T(g)=\bo \{b^*(v_0 g)<T\}$
denote the indicator function of the vector $v_0 g$ having bend at most $T$, where $g\in G$. This function is left-invariant under $H:=\Stab_G(v_0)=\{g\in G: v_0 g = v_0\}$, and also right-invariant under $L:=\Stab_{b^*}=\{g \in G : b^*(v g)=b^*(v),\text{ for all }v\in V\}$.

It will be useful to decompose $G=O_Q^\circ(\R)$ (for $Q$ given by \eqref{eq:Qeg}) and its Lie algebra $\fg:=Lie(G)$ as follows. First we decompose $\fg$ as:
$$
\fg =\fh \oplus \bar\fu \oplus \fu \oplus \fm_1.
$$
Here $\fh=Lie(H)$,
and with
$$
M:=
\begin{pmatrix}
    1 & 0 & 0  \\
    0 & O(n) & 0 \\
    0 & 0 & 1 \\
  \end{pmatrix},
$$
we have that $M\cap H \cong O(n-1)$. Then set $\fm:=Lie(M)$ and $\fm_1:=\fm\ominus Lie(M\cap H)$.
(Note that $\fh \cap \fm)$ is trivial when $G\cong O(3,1)$.)
Set $M_1:=\exp\fm_1<G$.
The one-parameter Lie algebras $\fu$ and  $\bar \fu$ are given by:
$$
\fu=
\left\{
\begin{pmatrix}
    0 & 0_{n-1} & w & 0  \\
    0 & 0 & 0 &0_{n-1} \\
    0 & 0 & 0 & 2w\\
    0 & 0
    & 0 & 0 \\
  \end{pmatrix}
  : w\in\R
  \right\},
$$
and
$$
\bar\fu=
\left\{
\begin{pmatrix}
    0 & 0 & 0 & 0  \\
    0_{n-1} & 0 & 0 &0 \\
    y & 0 & 0 & 0\\
    0 & 0_{n-1} & y/2 & 0 \\
  \end{pmatrix}
  : y\in\R
  \right\}
  ,
$$
which exponentiate to the groups
$$
U:=
\left\{
u(w):= \begin{pmatrix}
    1 & 0_{n-1} & w & w^2  \\
    0 & I_{n-1} & 0 &0_{n-1} \\
    0 & 0 & 1 & 2w\\
    0 & 0 & 0 & 1 \\
  \end{pmatrix}
  : w\in\R
  \right\},
$$
and
$$
\oU:=
\left\{\overline{u}(y)
:=
\begin{pmatrix}
    1 & 0 & 0 & 0  \\
    0_{n-1} & I_{n-1} & 0 &0 \\
    y & 0 & 1 & 0\\
    y^2/4 & 0_{n-1} & y/2 & 1 \\
  \end{pmatrix}
  : y\in\R
  \right\}
,
$$
respectively.

This gives the corresponding decomposition:
$$
H\times \oU \times U \times M_1\to G : (h,\ou, u, m_1)\mapsto h\ou u m_1.
$$

Geometrically this decomposition is in fact very natural. $H$ is the stabilizer of $v_0$, corresponding to the sphere whose orbit we are considering. The action of $\overline{U}$ changes the bend of the sphere, while $UM_1$ changes the co-bend by moving the center via a polar coordinate description of the plane ($U$ as the radial coordinate and $M_1\cong O(n)/O(n-1)$ giving a rotation). 

Let $d:=\dim(\mathfrak{g})=(n+1)(n+2)/2$. Let $k:= \dim(\mathfrak{h})=n(n+1)/2$, and let $\ell:= \dim(\mathfrak m_1) = n-1$. Hence $d= k + \ell + 2$.

\subsection{An Explicit Basis.}\label{subsec:basis} \

Before proceeding, it will help our calculations to fix a basis for $\mathfrak h$ and $\mathfrak m_1$. 
To 
ease
notation, we
will
describe a basis in a special case when $n= 4$, so $d=15$, $k=10$ and $\ell = 3$; the general case will be clear by analogy. Take:
\begin{align*}
  \left.
X_1:=\begin{psmallmatrix}
 0 & 0 & 0 & 0 & 0 & 0 \\
 0 & 0 & 0 & 0 & 0 & 0 \\
 0 & 0 & 0 & 1 & 0 & 0 \\
 0 & 0 & -1 & 0 & 0 & 0 \\
 0 & 0 & 0 & 0 & 0 & 0 \\
 0 & 0 & 0 & 0 & 0 & 0 \\
\end{psmallmatrix}
,
X_2:=\begin{psmallmatrix}
 0 & 0 & 0 & 0 & 0 & 0 \\
 0 & 0 & 0 & 1 & 0 & 0 \\
 0 & 0 & 0 & 0 & 0 & 0 \\
 0 & -1 & 0 & 0 & 0 & 0 \\
 0 & 0 & 0 & 0 & 0 & 0 \\
 0 & 0 & 0 & 0 & 0 & 0 \\
\end{psmallmatrix}
,
X_3:=\begin{psmallmatrix}
 0 & 0 & 0 & 0 & 0 & 0 \\
 0 & 0 & 1 & 0 & 0 & 0 \\
 0 & -1 & 0 & 0 & 0 & 0 \\
 0 & 0 & 0 & 0 & 0 & 0 \\
 0 & 0 & 0 & 0 & 0 & 0 \\
 0 & 0 & 0 & 0 & 0 & 0 \\
\end{psmallmatrix},
\right\} \mbox{  $\in \ \mathfrak h \cap \mathfrak m$ }
\end{align*}

\begin{align*}
  \left.
  X_4:=\begin{psmallmatrix}
 0 & 0 & 0 & 0 & 0 & 0 \\
 0 & 0 & 0 & 0 & 0 & 0 \\
 0 & 0 & 0 & 0 & 0 & 0 \\
 1 & 0 & 0 & 0 & 0 & 0 \\
 0 & 0 & 0 & 0 & 0 & 0 \\
 0 & 0 & 0 & \frac{1}{2} & 0 & 0 \\
\end{psmallmatrix}
,
X_5:=\begin{psmallmatrix}
 0 & 0 & 0 & 0 & 0 & 0 \\
 0 & 0 & 0 & 0 & 0 & 0 \\
 1 & 0 & 0 & 0 & 0 & 0 \\
 0 & 0 & 0 & 0 & 0 & 0 \\
 0 & 0 & 0 & 0 & 0 & 0 \\
 0 & 0 & \frac{1}{2} & 0 & 0 & 0 \\
\end{psmallmatrix}
,
X_6:=\begin{psmallmatrix}
 0 & 0 & 0 & 0 & 0 & 0 \\
 1 & 0 & 0 & 0 & 0 & 0 \\
 0 & 0 & 0 & 0 & 0 & 0 \\
 0 & 0 & 0 & 0 & 0 & 0 \\
 0 & 0 & 0 & 0 & 0 & 0 \\
 0 & \frac{1}{2} & 0 & 0 & 0 & 0 \\
\end{psmallmatrix},
\right\} \mbox{Lower Triangular in $\mathfrak h$ }
\end{align*}

\begin{align*}
\left.
X_7:=\begin{psmallmatrix}
 0 & 0 & 0 & 1 & 0 & 0 \\
 0 & 0 & 0 & 0 & 0 & 0 \\
 0 & 0 & 0 & 0 & 0 & 0 \\
 0 & 0 & 0 & 0 & 0 & 2 \\
 0 & 0 & 0 & 0 & 0 & 0 \\
 0 & 0 & 0 & 0 & 0 & 0 \\
\end{psmallmatrix}
,
X_8:=\begin{psmallmatrix}
 0 & 0 & 1 & 0 & 0 & 0 \\
 0 & 0 & 0 & 0 & 0 & 0 \\
 0 & 0 & 0 & 0 & 0 & 2 \\
 0 & 0 & 0 & 0 & 0 & 0 \\
 0 & 0 & 0 & 0 & 0 & 0 \\
 0 & 0 & 0 & 0 & 0 & 0 \\
\end{psmallmatrix}
,
X_9:=\begin{psmallmatrix}
 0 & 1 & 0 & 0 & 0 & 0 \\
 0 & 0 & 0 & 0 & 0 & 2 \\
 0 & 0 & 0 & 0 & 0 & 0 \\
 0 & 0 & 0 & 0 & 0 & 0 \\
 0 & 0 & 0 & 0 & 0 & 0 \\
 0 & 0 & 0 & 0 & 0 & 0 \\
\end{psmallmatrix}
,
\right\} \mbox{Upper Triangular in $\mathfrak h$ }
\end{align*}
\begin{align*}
  \left.
  X_{10}:=\begin{psmallmatrix}
 1 & 0 & 0 & 0 & 0 & 0 \\
 0 & 0 & 0 & 0 & 0 & 0 \\
 0 & 0 & 0 & 0 & 0 & 0 \\
 0 & 0 & 0 & 0 & 0 & 0 \\
 0 & 0 & 0 & 0 & 0 & 0 \\
 0 & 0 & 0 & 0 & 0 & -1 \\
\end{psmallmatrix}
,
\right. 
\mbox{ -- Diagonal in $\mathfrak h$}
\end{align*}
\begin{align*}
X_{11}:=\begin{psmallmatrix}
 0 & 0 & 0 & 0 & 0 & 0 \\
 0 & 0 & 0 & 0 & 0 & 0 \\
 0 & 0 & 0 & 0 & 0 & 0 \\
 0 & 0 & 0 & 0 & 0 & 0 \\
 1 & 0 & 0 & 0 & 0 & 0 \\
 0 & 0 & 0 & 0 & \frac{1}{2} & 0 \\
\end{psmallmatrix} 
\ \in \  \overline{\mathfrak{u}}
, \
X_{12}:=\begin{psmallmatrix}
 0 & 0 & 0 & 0 & 1 & 0 \\
 0 & 0 & 0 & 0 & 0 & 0 \\
 0 & 0 & 0 & 0 & 0 & 0 \\
 0 & 0 & 0 & 0 & 0 & 0 \\
 0 & 0 & 0 & 0 & 0 & 2 \\
 0 & 0 & 0 & 0 & 0 & 0 \\
\end{psmallmatrix}
\ \in \ \mathfrak u,
\end{align*}
\begin{align*}
  \left.
X_{13}:=\begin{psmallmatrix}
 0 & 0 & 0 & 0 & 0 & 0 \\
 0 & 0 & 0 & 0 & 0 & 0 \\
 0 & 0 & 0 & 0 & 0 & 0 \\
 0 & 0 & 0 & 0 & 1 & 0 \\
 0 & 0 & 0 & -1 & 0 & 0 \\
 0 & 0 & 0 & 0 & 0 & 0 \\
\end{psmallmatrix}
,
X_{14}:=\begin{psmallmatrix}
 0 & 0 & 0 & 0 & 0 & 0 \\
 0 & 0 & 0 & 0 & 0 & 0 \\
 0 & 0 & 0 & 0 & 1 & 0 \\
 0 & 0 & 0 & 0 & 0 & 0 \\
 0 & 0 & -1 & 0 & 0 & 0 \\
 0 & 0 & 0 & 0 & 0 & 0 \\
\end{psmallmatrix}
,
X_{15}:=\begin{psmallmatrix}
 0 & 0 & 0 & 0 & 0 & 0 \\
 0 & 0 & 0 & 0 & 1 & 0 \\
 0 & 0 & 0 & 0 & 0 & 0 \\
 0 & 0 & 0 & 0 & 0 & 0 \\
 0 & -1 & 0 & 0 & 0 & 0 \\
 0 & 0 & 0 & 0 & 0 & 0 \\
\end{psmallmatrix}
\right\} \qquad  \in \  \mathfrak m_1.
\end{align*}

To move back to group elements we denote
\begin{align*}
    &h_j(x_j):=\exp(x_j X_j) \ \in \ H \ \qquad\qquad \mbox{ for } j \le k,\\
    &\overline{u}(y):=\exp(y X_{k+1}) \ \in \ \overline U,\\
    &u(w):=\exp(w X_{k+2}) \ \in \ U,\\
    &m_j(\varphi_j):=\exp(\varphi_j X_{j+k+2}) \in \ M_1  \qquad \mbox{ for } 1 \le j \le \ell.
\end{align*}

It will also be convenient to record the following components of $H$:
\be\label{eq:Hparts}
H_M := \prod_{X_j\in\mathfrak{h}\cap\mathfrak{m}}\exp(x_j X_j),\
H_+:=\prod_{X_j\text{ ``Upper Triangular''}}\exp(x_j X_j),\ 
\ee
$$
H_-:=\prod_{X_j\text{ ``Lower Triangular''}}\exp(x_j X_j),\ \text{and}\
H_A:=\prod_{X_j\text{ ``Diagonal''}}\exp(x_j X_j).
$$

\subsection{Calculating the Haar Measure}
\label{ss:Haar}

Given that decomposition of the Lie algebra, we now derive an explicit form of the Haar measure in our chosen coordinate system. Let
\begin{align*}
    \vect{z}:= (x_1, \dots , x_k, y, w, \varphi_1, \dots, \varphi_\ell)
\end{align*}
denote a set of coordinates and let 
\begin{align*}
    \cJ : \R^d \to G : \vect{z} \mapsto h_1(x_1)h_2(x_2) \dots h_k(x_k) \overline{u}(y) u(w) m_1(\varphi_1) \dots m_{\ell}(\varphi_\ell)
\end{align*}
map our coordinate space to $G$. Then, in these new coordinates we have the following decomposition of the Haar measure

\begin{theorem} \label{thm:rho decomp} [Haar Measure Structure Theorem]
  Let $\rho(\vect{z})$ denote the density of the Haar measure in $\vect{z}$ coordinates. Then 
  \begin{align}\label{rho decomp}
      \rho(\vect{z}) = \rho_H(x_1, \dots x_k) \rho_{\oU U}(y,w) \rho_{M_1}(\varphi_1, \dots, \varphi_{\ell}),
  \end{align}
  where $\rho_H$ and $\rho_{M_1}$ are (respectively) the densities for the Haar measure on $H$ and $M_1$. Moreover we have that $\rho_{\oU U} = \abs{1+wy}^{n-1}$. 
\end{theorem}

\begin{proof}

Given a linear differential operator on $\vect{z}$ we can represent it as 
\begin{align*}
    T_{\vect{z}} = \eta_1 \partial_{x_1}  +\dots + \eta_k \partial_{x_k} +\eta_{k+1} \partial_{y} + \eta_{k+2} \partial_{w} +  \eta_{k+3} \partial_{\varphi_{1}}  +\dots + \eta_{d} \partial_{\varphi_\ell}.
\end{align*}
Now using the correspondence between elements of the Lie algebra and linear differential operators we write
\begin{align*}
    T_{\vect{z}} = \eta_1 X_1 + \eta_2 X_2 +\dots +\eta_{d}X_{d}.
\end{align*}
To first order, we can write the differential $T_{\vect{z}}
$ acting on an element of $g$ as
\begin{align*}
  & h_1(I+\eta_1X_1) \cdots h_k(I+\eta_kX_k) \overline{u}(I+\eta_{k+1}X_{k+1})\\
  &\phantom{===} \cdot u(I+\eta_{k+2}X_{k+2}) m_1 (I+\eta_{k+3}) \cdots m_{\ell}(I+\eta_{d})\\
    &\phantom{=}=
    h_1\cdots h_k \overline{n}n m_1 \cdots m_\ell \left\{ I + \ad((h_2\cdots h_k\overline{u}u m_1 \cdots m_\ell)^{-1},\eta_1X_1)   \right.\\
    &\phantom{==} \left.
    + \ad((h_3\cdots h_k\overline{u}u m_1 \cdots m_\ell)^{-1},\eta_2X_2) + \dots + \ad(m_\ell^{-1},\eta_{d-1} X_{d-1}) + \eta_{d} X_{d}  
    \right\}.
\end{align*}
Giving us the differential operator
\begin{align*}
    &D(\cJ):\eta_1 \partial_{x_1}  +\dots + \eta_k \partial_{x_k} +\eta_{k+1} \partial_{y} + \eta_{k+2} \partial_{w} +  \eta_{k+3} \partial_{\varphi_{1}}  +\dots + \eta_{d} \partial_{\varphi_\ell}\\
    &\phantom{+} \mapsto \eta_1\ad((h_2\cdots h_k\overline{u}u m_1 \cdots m_\ell)^{-1},X_1) 
    + \\&\eta_2\ad((h_3\cdots h_k\overline{u}u m_1 \cdots m_\ell)^{-1},X_2) + \dots
      +\eta_{k+\ell+1}\ad(m_\ell^{-1}, X_{d-1}) + \eta_{d} X_{d}.
\end{align*}
Hence, if we want to apply the differential operator $X_j$ on the right to an element $g$, then we simply solve for $\eta_i$ on the right hand side of this map. Then the left hand side tells us the action on the coordinates. Let us denote the vector of such $\eta_i$'s by $(\eta_{j1}, \eta_{j2},\dots, \eta_{jd})$. Thus
\begin{align*}
     X_j &=\eta_{j1}\ad((h_2\cdots h_k\overline{n}n m_1 \cdots m_\ell)^{-1},X_1) 
     + \eta_{j2}\ad((h_3\cdots h_k\overline{u}u m_1 \cdots m_\ell)^{-1},X_2)\\
     &+ \dots
    +\eta_{j(d-1)}\ad(m_\ell^{-1}, X_{d-1}) + \eta_{jd} X_{d},
\end{align*}
for every $j = 1, \dots, d$.

Now to calculate the Haar measure, we proceed by  the following methodology: define the right multiplication operator $R_A(\vect{z}) := \cJ^{-1}( \cJ(\vect{z})\cdot A)$ and its Jacobian:
  \begin{align*}
    \left[R_A^\prime(\vect{z})\right]_{jk} := \frac{\partial^jR_A(\vect{z})}{d\vect{z}_k}.
  \end{align*}
  Then our goal is to find $\rho(\vect{z})$ such that:
  \begin{align}\label{app 1}
    \int f(\cJ(\vect{z}))\rho(\vect{z}) d\vect{z} =  \int f(\cJ(R_{A}(\vect{z})))\rho(\vect{z}) d\vect{z}.
  \end{align}
Changing variables $\vect{y} = R_{A^{-1}}(\vect{z})$, on the left hand side gives:
  \begin{align*}
    \int f(\cJ(\vect{z}))\rho(\vect{z}) d\vect{z} =  \int f(\cJ(R_{A}(\vect{y})))\rho(R_A(\vect{y})) \abs{\det R_A^\prime(\vect{y})}d\vect{\vect{y}}
  \end{align*}
  which is equal to the right hand side of \eqref{app 1}. Hence we want to find $\rho$ such that:
  \begin{align*}
    \rho(R_A(\vect{y})) = \frac{\rho(\vect{y})}{\abs{\det R_A^\prime(\vect{y})}}
  \end{align*}
  for all choices of $\vect{y}$ and $A\in G$. In particular, we may choose $\vect{y} = 0$, and $A = \cJ(\vect{z})$, which gives
  \begin{align}
    \rho(\vect{z}) = \frac{1}{\abs{\det R^\prime_{\cJ(\vect{z})}(0)}}.
  \end{align}
  Note that since the Haar measure is only unique up to a constant, we set $\rho(0)=1$ without loss of generality. Using our decomposition of the Lie algebra we can write:
  \begin{align*}
    [R_{\cJ(\vect{z})}^\prime(0)]_{ij} = \left[\frac{\partial}{\partial_{t_i}} \cJ(\vect{z}) e^{ \sum_{i=1}^d t_i X_i}\Bigg|_{\vect{t} = 0} \right]_j 
  \end{align*}
  Now we can linearize the exponential and then find the corresponding coordinate description of the differential operator as we did above. The above argument implies
  \begin{align*}
    \rho(\vect{z})^{-1} = \det\left|[R_{\cJ(\vect{z})}^\prime(0)]_{ij}\right| = \det|\eta_{ij}|.
  \end{align*}
  
  Now to simplify matters, we express each adjoint as a linear combination of elements in the basis:
  \begin{align*} 
    \ad((h_2\cdots h_k\overline{u}u m_1 \cdots m_\ell)^{-1},X_1) 
    &= \sum_{i=1}^{d} \mu_{1i} X_i \\
    \ad((h_3\cdots h_k\overline{u}u m_1 \cdots m_\ell)^{-1},X_2) 
    &= \sum_{i=1}^{d} \mu_{2i} X_i \\
    &\dots\\
    X_{d} 
    &= \sum_{i=1}^{d} \mu_{di} X_i.
\end{align*}
Thus, to calculate the Haar measure, by linearity of the adjoint operator, we need to find a $d\times d$ matrix $\eta$ such that
\begin{align*}
    \eta \mu \vect{X} = \vect{X}
\end{align*}
where $\vect{X} : = (X_1, X_2, \dots X_{d})^T$. Hence, since the determinant is multiplicative, we have that $\rho(\vect{z})= \det[\mu]$.

Now suppose we wanted to calculate the Haar measure of $H$, then we can write
\begin{align*}
    \ad((h_2\cdots h_k)^{-1},X_1) 
    &= \sum_{i=1}^{d} \nu_{1i} X_i \\
    \ad((h_3\cdots h_k)^{-1},X_2) 
    &= \sum_{i=1}^{d} \nu_{2i} X_i \\
    &\dots\\
    \ad(h_k^{-1},X_{k-1})
    &= \sum_{i=1}^{d} \nu_{(k-1)i} X_i \\
    X_{k} 
    &= \sum_{i=1}^{d} \nu_{ki} X_i.
\end{align*}
For completeness write $\nu_{ji} = \delta_{i,j}$ for $j > k$. Then the above argument implies that one can write the Haar measure on $H$ as $\rho_H(\vect{z}) = \det[\nu]$. 

Now the key observation is to write
\begin{align*}
    \ad((\overline{u}u m_1 \cdots m_\ell)^{-1},X_1) 
    &= \sum_{i=1}^{d} c_{1i} X_i \\
    &\dots \\
    \ad((\overline{u}u m_1 \cdots m_\ell)^{-1},X_{k}) 
    &= \sum_{i=1}^{d} c_{ki} X_i \\
    \ad((u m_1 \cdots m_\ell)^{-1},X_{k+1}) 
    &= \sum_{i=1}^{d} c_{(k+1)i} X_i \\
    &\dots\\
    X_{d} 
    &= \sum_{i=1}^{d} c_{di} X_i.
\end{align*}
By linearity of the adjoint operator we have that $\mu=\nu c $, and hence $\det(\mu) = \det(\nu)\det(c)$. Or rather, $\rho(\vect{z}) = \rho_H(x_1, \dots, x_k) f(w,y,\varphi_1, \dots, \varphi_\ell)$ for some function $f$. 

Moreover, since the Haar measure is unimodular, and since $M_1$ is a group, we can apply the same argument to $M_1$ on the right, and show that $\rho$ satisfies the product structure from \eqref{rho decomp}.

Finally, thanks to this product structure, we can take $m_1, \dots, m_l$ equal to the identity when calculating $\rho_{\oU U}$. Thus, let
\begin{align}\label{d def}
\begin{aligned}
    \ad((\overline{u}u)^{-1},X_j) 
    &= \sum_{i=1}^{d} d_{ji} X_i, \qquad \mbox{ for }j =1 , \dots, k \\
    \ad((u)^{-1},X_{k+1}) 
    &= \sum_{i=1}^{d} d_{(k+1)i} X_i \\
    X_{j} 
    &= \sum_{i=1}^{d} d_{ji} X_i, \qquad \mbox{ for }j=k+2, \dots, d.
    \end{aligned}
\end{align}
Hence it remains to find $\det(d)$.

No matter the dimension, our basis elements are each one of seven types (in $\mathfrak{m}\cap \mathfrak{h}$, lower triangular in $\mathfrak{h}$, upper triangular in $\mathfrak{h}$, diagonal, in $\overline{\mathfrak{u}}$, in $\mathfrak{u}$, or in $\mathfrak{m}_1$). Depending on the type of $X_i$, we can calculate $d_{ij}$ explicitly independent of dimension via an inductive argument. From which it follows that the $d$ matrix has the following form:
\begin{align}\label{d explicit}
    d = \begin{pmatrix}
       I_{(n-1)(n-2)/2} & 0 & 0 & 0 & 0\\
       0 & I_{n-1} &\frac{w^2}{2} I_{n-1} & 0 & w I_{n-1}\\
       0 & \frac{y^2}{2} I_{n-1} & \frac{1}{4}(2+wy)^2 I_{n-1} & 0 & \frac{1}{2}w(2+wy)I_{n-1}\\
       0 & 0 & 0 & B & 0 \\
       0 & 0 & 0 & 0 & I_{n-1}
    \end{pmatrix}
\end{align}
where $B$ is the $3\times 3$ matrix given by
\begin{align*}
    B:= \begin{pmatrix}
    1+wy & -w & \frac{1}{2}(2+wy)\\
    -y & 1 & -\frac{y^2}{2} \\
    0 & 0 & 1.
    \end{pmatrix}
\end{align*}
Hence, one can determine explicitly thanks to a block matrix decomposition that $\det(d)= \abs{1+wy}^{n-1}$.
This completes the proof.
\end{proof}


\subsection{Calculating the Casimir Operator}
\label{ss:Casimir}

As in \S\ref{sec:prelim}, given our basis, $X_1,X_2 \dots, X_{d}$,  for the Lie algebra in \S\ref{subsec:basis}, there is a corresponding dual basis $X_1^\ast, X_2^\ast, \dots X_{d}^\ast$ with respect to the Killing form. Then the Casimir operator can be written as the following second order differential operator:
\begin{align}
    \cC : = \sum_{i=1}^{k+2+\ell} X_i X_i^\ast.
\end{align}
Using the above argument, in $\vect{z}$ coordinates, we can express $X_i = \sum_{j=1}^{d}\eta_{ij} \partial_j$. Likewise, we can express $X_i^\ast = \sum_{j=1}^{d}\eta_{ij}^\ast \partial_j$. 

Now define $\mu^\ast$ analogously to how we defined $\mu$, that is
 \begin{align*} 
    \ad((h_2\cdots h_k\overline{u}u m_1 \cdots m_\ell)^{-1},X_1) 
    &= \sum_{i=1}^{d} \mu_{1i}^\ast X_i^\ast \\
    \ad((h_3\cdots h_k\overline{u}u m_1 \cdots m_\ell)^{-1},X_2) 
    &= \sum_{i=1}^{d} \mu_{2i}^\ast X_i^\ast \\
    &\dots\\
    X_{d}&= \sum_{i=1}^{d} \mu_{di} X_i^\ast.
\end{align*}
We then have that $\eta^\ast \cdot \mu^\ast \vect{X}^\ast = \vect{X}^\ast$, and hence $\eta^\ast = (\mu^\ast)^{-1}$. Thus, if we define $\underline{\partial} := (\partial_{x_1}, \dots , \partial_{n_{d}})^T$, then as a differential operator \begin{align*}
    \cC = \left( \mu^{-1} \underline{\partial}\right) \cdot \left((\mu^{\ast})^{-1}\underline{\partial}\right).
\end{align*}
The Casimir operator in all $d$ variables is, of course, rather unwieldy. However, fortunately we shall only need the Casimir operator to act on left $H$-invariant, and right $M_1$-invariant functions. Hence, if we write $\mu^\ast = \nu^\ast c^\ast$, with $\nu^\ast$ and $c^\ast$ defined analogously to as above, then when acting on such functions, we have
\begin{align*}
    \cC = \left( c^{-1} \underline{\partial}\right) \cdot \left((c^\ast)^{-1}\underline{\partial}\right).
\end{align*}
Using this decomposition, we can explicitly calculate the Casimir operator in $\vect{z}$ coordinates, when acting on left $H$-invariant and right $M_1$-invariant function. 
\begin{theorem}\label{thm:Casimir}[Structure Theorem for the Casimir Operator]
    Let $f : H\backslash G / M_1 \to \C$. Then in the $\vect{z}$-coordinate system, $f$ is only a function of $(y,w)$, that is, the $\overline {U}$ and $U$ variables, and the Casimir operator acting on $f$ has the following form:
    \begin{align}\label{Cas explicit}
        \cC f(y,w) =  \frac{1}{2}\left(y^2 \partial_{y}^2 + (n+1) y \partial_{y} + 2\partial_{yw} + \frac{(n-1) w \partial_{w}}{1+yw}\right)f(y,w).
    \end{align}
\end{theorem}

\begin{proof}
    By definition of $c^\ast$ we have
    \begin{align*}
    Y_j : = \ad((\overline{u}u m_1 \cdots m_\ell)^{-1},X_j) 
    &= \sum_{i=1}^{d} c^\ast_{ki} X_i^\ast \ \mbox{for }i=1, \dots, k
    \end{align*}
    and
    \begin{align*}
    Y_{k+1}:=\ad((u m_1 \cdots m_\ell)^{-1},X_{k+1}) 
    &= \sum_{i=1}^{d} c^\ast_{(k+1)i} X_i^\ast \\
    &\dots
    \\
    Y_{d}:=X_{d} 
    &= \sum_{i=1}^{d} c_{di} X_i^\ast.
    \end{align*}
    Thus, $(c^{\ast})^{-1}\vect{Y} = \vect{X}$. Since our function depends only on $w$ and $y$, we are only interested in the $(k+1)^{th}$ and $(k+2)^{th}$ columns of $(c^{\ast})^{-1}$. We can explicitly calculate $Y_i$ depending on which type of basis vector is $X_i$, then we can explicitly solve for $(c^{\ast})^{-1}_{ij}$. Using an inductive argument we then have that $(c^\ast)^{-1} \underline{\partial}$ has the following explicit form: Let 
  \begin{align}\label{vM}
    v_M&:=(\sin(\varphi_1),\cos(\varphi_1)\sin(\varphi_2),\cos(\varphi_1)\cos(\varphi_2)\sin(\varphi_3),\\
    \notag &\phantom{+++++++++++}\dots, \cos(\varphi_1) \cdots \cos(\varphi_{\ell-1})\sin(\varphi_\ell))^T
  \end{align}
  then we have (in fact, we only need the fourth and fifth row of the below)
  \begin{align}\label{cdual diff}
  \begin{aligned}
      (c^\ast)^{-1} \underline{\partial} = \frac{1}{2}
      \begin{pmatrix}
      0_{(n-1)n/2}\\
      -v_M\partial_{w}\\
      \frac{1}{2}v_M(w^2 \partial_{w} - 2(1+ wy )\partial_{y})\\
      (y\partial_{y} - w\partial_{w})\\
      \cos(\varphi_1)\dots \cos(\varphi_\ell) \partial_{w}\\
      \frac{1}{2}
      \cos(\varphi_1)\dots \cos(\varphi_\ell)(-w^2\partial_{w} + 2(1+ wy)\partial_{y})\\
      0_{n-2}
      \end{pmatrix}
      \end{aligned}.
  \end{align}

  Let $(A_1, \dots, A_d) : = (c^\ast)^{-1}\cdot\underline{\partial}$, then our aim is to evaluate
  \begin{align*}
    \cC=
      \sum_{j=1}^d \sum_{i=1}^d (c^{-1})_{ji} \partial_i A_j.
  \end{align*}
  However, since we are considering the Casimir acting on right $M_1$ invariant functions, 
  we may set the last $\ell$ coordinates equal to $0$, that is
  \begin{align*}
    \cC f(y,w)=
      \left(\sum_{j=1}^d \sum_{i=1}^d \left[(c^{-1})_{ji}\right]_{\varphi=0} \left[\partial_i A_j \right]_{\varphi=0}   \right)f(y,w),
  \end{align*}
  however $\left[(c^{-1})_{ji}\right]_{\varphi=0} = (d^{-1})_{ji}$ where $d$ is the matrix of coefficients defined in \eqref{d def}. 
  
  Since the $d$ matrix is explicitly described in \eqref{d explicit}, and the vector $(A_1, \dots, A_d)$ is given explicitly in \eqref{cdual diff}, then \eqref{Cas explicit} can be found through direct computation. 
\end{proof}

\section{Counting}
\label{s:Counting}
We turn now to the proof of Theorem \ref{thm:main}. For simplicity, we assume the packing is bounded; similar methods apply for counting in regions.
We also assume that the packing is the orbit of a single sphere $S_1$ under the action of a symmetry group $\G$; in general, the counting problem reduces to a finite sum of such \cite[Theorem 22]{KapovichKontorovich2021}.
Let the sphere $S_1$ be represented 
in the inversive coordinate system described in \S\ref{sec:ACP} by the vector $v_1\in V$.
We begin in the same way as we did in the $\SL_2(\R)$ case, writing the count as
\begin{align*}
  N_\cP (T) &:= \# \left\{ S \in \cP  \; |\; b(S) < T \right\}\\
           &= \# \left\{ v \in v_1 \cdot \Gamma \; |\; b^* (v) < T \right\}\\
           &= \sum_{\gamma \in \Gamma_1\backslash \Gamma } \chi_T(\gamma),
\end{align*}
where  $b^*$ is the ``bend'' covector (as in \eqref{eq:bendIs}),  $\Gamma_1:=\Stab_{v_1}(\Gamma)$, and for $g\in G$
\begin{align}
  \chi_T( g ) : = \begin{cases}
    1 & \mbox{ if } b^*(v_1 \cdot g) < T,\\
    0 & \mbox{ else. }
  \end{cases}
\end{align}
We observe that
$\chi_T$ is a function on $H\bk G / UM_1$,
where
$H$ is the stabalizer of $v_1$ (so that $\G_1=\G\cap H$). 

Now we automorphize $\chi_T$, that is define
\begin{align*}
    F_T(g): = \sum_{\gamma \in \Gamma_1 \bk \Gamma} \chi_T(\gamma g).
\end{align*}
Hence $F_T$ is left $\Gamma$-invariant and $N_\cP(T) = F_T(e)$. Again, rather than trying to evaluate the discontinuous function $F_{T}$ at the origin, consider its inner product with a smooth approximation of the identity. However,  we want to restrict as few directions of our smooth approximation as possible, to optimize the resulting error terms.
To this end, fix an $\vareps>0$, and let
\begin{align}
  \psi = \psi_{\vareps} \in L^2(\Gamma_1\backslash G/M_1 )
\end{align}
be given as follows.
Let $\cF$ be a fundamental domain described in $\vect{z}$-coordinates, of $\Gamma_1 \bk G /M_1$. Then on $\cF$ we let $\psi$ be of the form
 $$\psi(\vect{z}) = \psi_1(x_1)\cdots\psi_k(x_k)\psi_{U}(y)\psi_{\oU}(w)$$, with all components nonnegative and unit total mass,
 $$
 \int_{\G_1\bk G/M_1}\psi dg = 1,
 $$
 and satisfying the following conditions. 
 
For the $n+1$ variables in the components $H_+$, $H_A$ (in the notation of \eqref{eq:Hparts}) and $\overline{U}$, we restrict the coefficients to $\vep$-balls around $0$. The other variables are restricted only to compact regions of bounded size around $0$.
We can choose such a $\psi$ have $L^2$ mass bounded by:
$$
\|\psi\|_{L^2(\G_1\bk G/M_1)} \ll \vep^{-(n+1)/2}.
$$


Now define the function $w_{T} = w_{T,\vareps}: \Gamma_1\backslash G \to [0,1]$ by
\begin{align}
  w_T(g):= \int_{\Gamma_1\backslash G} \chi_T(h) \psi_\vareps(g h) d h.
\end{align}
Our main counting theorem (Theorem \ref{thm:main}) will follow from the following smooth effective count.

\begin{theorem} \label{thm:w count}
  Let the Laplace eigenvalues of $\G\bk \half^{n+1}$ be as in \eqref{eq:gls}.
  Then there exist constants $c^i_{\Gamma,\vareps}$ for $i=0,1,\dots, k,$ with $c^0_{\Gamma,\vareps}>0$, such that
  \begin{align}\label{w asymp}
    \wt{N}_{\Gamma,\vep}(T)&:=\sum_{\gamma \in \Gamma_1\backslash \Gamma} w_T(\gamma) = c_{\Gamma,\vareps}^0T^\delta+c_{\Gamma,\vareps}^1T^{s_1}\\
    \notag &\phantom{+++++++}+\dots+c_{\Gamma,\vareps}^kT^{s_k} + O \left(\frac{1}{\vareps^{(n+1)/2}}T^{n/2}\log(T)\right)
  \end{align}
  where the implied constant depends only on $\Gamma$. Moreover $c^0_{\Gamma,\vareps} = c^0_\Gamma(1+ O(\vareps))$, and for $i\ge1$ (if any exist), we have that $c^i_{\Gamma,\varepsilon}\ll \varepsilon^{-(n+1)/2}$.

\end{theorem}

\begin{remark}
  Note that this smoothed counting theorem is optimal in that the error goes all the way to the tempered spectrum.
\end{remark}

\subsection{Proof of Theorem \ref{thm:main}}

An explicit calculation shows that the inversive coordinates of a sphere which has been transformed by $g$ are
\begin{align*}
    (y,-(1+wy)v_M^T, (1+wy)\cos(\varphi_1)\cdots\cos(\varphi_\ell), w (2+wy))
\end{align*}
where $v_M$ is the vector defined in \eqref{vM}. Thus the `bend radius' (i.e the distance of the center of the sphere to the origin multiplied by the bend) is $(1+wy)$. Since this distance is necessarily bounded from above (since the packing is bounded) and since $y$ is  bounded from below (because it controls the bend, and the packing has a largest sphere), we know that $w$ is bounded from above. In particular we have

\begin{lemma}\label{lem:geometric}
  Let $\gamma \in \Gamma_1 \bk \Gamma$ then $\abs{w_\gamma}$ is bounded. 
\end{lemma}

Now, assuming Theorem \ref{thm:w count} we present the proof of Theorem \ref{thm:main}.

\begin{proof}[Proof of Theorem \ref{thm:main}]
  Consider

  \begin{align*}
    w_T(\gamma) &= \int_{\Gamma_1\backslash G} \chi_T(g) \psi_\vareps(\gamma^{-1} g) d g \\
    &= \int_{\Gamma_1\backslash G} \chi_T(\gamma g)\psi_\vareps(g) d g.
  \end{align*}
  We write $\gamma= \cJ(x_{1,\gamma},\dots, x_{k,\gamma}, y_\gamma,w_\gamma,\varphi_{1,\gamma}, \dots, \varphi_{\ell, \gamma})$ and note that $\chi_T$ is left $H$-invariant and right $UM_1$ invariant. Hence
  \begin{align*}
      \chi_T(\gamma g)  = \chi_T( \ou(y_\gamma)u(w_\gamma) m_1(\vect{\varphi}_\gamma) h(\vect{x}) \on(y) ).
  \end{align*}
  now we can write $h$ as a product of one dimensional components of $H$ from \eqref{eq:Hparts} (here, it is convenient to change the order of the decomposition). That is, write 
  $$h(\vect{x}) =  h_+(\vect{x}_+)  h_A(\vect{x}_A)h_-(\vect{x}_-)h_M(\vect{x}_M).$$ 
  
  Since $\psi$ restricts the $\vect{x}_-$ and $\vect{x}_A$ coordinates to balls of radius $\vep$, we can apply adjoints to move those factors to the left of $\ou(y_\gamma)$ and the other factors will be perturbed by an $\vep$ error. Then we can use left $H$-invariance of $\chi_T$. We can do the same for the $\ou(y)$ factor which gets absorbed in the $\ou(y)$ factor. Thus, we have that
  \begin{align*}
      \chi_T(\gamma g)  &= \chi_T( \ou(y_\gamma + O(\vep))u(w_\gamma+ O(\vep)) m_1(\vect{\varphi}_\gamma+ O(\vep)) h_{+}(\vect{x}_+ +O(\vep))h_{M}(\vect{x}_M+O(\vep) )\\
      &= \chi_T( \ou(y_\gamma + O(\vep)) ),
  \end{align*}
  where the second line follows from right-$M$ invariance, and right-$H_+$ invariance.
  
 Thus,
 \begin{align*}
    \chi_T(\gamma g) = \begin{cases}
      1 & \mbox{ if } y_\gamma<  \frac{T}{1 + c\vareps }\\
      0 & \mbox{ if } y_\gamma>  \frac{T}{1 - c\vareps }
    \end{cases}
  \end{align*}
  for some absolute constant $c$.

  Hence, since $\wt{N}_{\Gamma,\vep}(T) = \sum_{\Gamma_1\backslash \Gamma} w_{T}(\gamma)$, then our count  satisfies:
  \begin{align}\label{Count eps bounds}
    \wt{N}_{\Gamma,\vep}(T(1-c\vareps)) \le N^{\cP}(T) \le \wt{N}_{\Gamma,\vep}(T(1+c\vareps))
  \end{align}
  Now, assuming there are no other eigenvalues, we apply Theorem \ref{thm:w count} to find:
  \begin{align*}
    \wt{N}_{\Gamma,\vep}(T(1\pm\vareps)) = c^0(1+O(\vareps))T^{\delta} + O(\frac{1}{\vareps^{(n+1)/2}}T^{n/2} \log(T)).
  \end{align*}
  Then choosing $\vareps = T^{\frac{2}{n+3}(n/2-\delta)}\log(T)^{2/(n+3)}$ optimizes this inequality. Thus
  \begin{align*}
    \wt{N}_{\Gamma,\vep}(T(1\pm\vareps)) = c^0T^{\delta} + O(T^{\frac{2}{n+3}(n/2-\delta)+\delta}\log(T)^{\frac{2}{n+3}}).
  \end{align*}
    The general case follows similarly.
\end{proof}

\subsection{Inserting the Casimir Operator}

Once again the smooth count is the inner product of $F_T$ with an $L^2$ function. Consider the inner product of $F_T$ with a general $\Psi\in L^2(\Gamma \bk G /M_1)$. That is, given a function $\psi$ on $\Gamma_1 \bk G/M_1$ we automorphize it
$$
    \Psi(z) := \sum_{\g\in\Gamma_1\bk \G}\psi(\g g).
$$
Let
\begin{align*}
  K_T(s) : = \frac{T^{s}b^{n-s} - T^{n-s} b^s}{b^{n-s}-b^s}, \qquad\qquad
  L_T(s) : = \frac{ T^{n-s}-T^s}{b^{n-s}-b^s},
\end{align*}
while for $s = n/2 +it$ we have 
\begin{align}\label{KL def ACP}
  K_T(s) : = T^{n/2}\frac{\sin(t\log T/\log b)}{\sin(t \log b)}, \qquad\qquad
  L_T(s) : = \left(\frac{T}{b}\right)^{n/2}\frac{\sin(t\log T)}{\sin(t \log b)}.
\end{align}
Again, by choosing an appropriate choice of $b$ one can ensure that
\begin{align}\label{K and L bounds ACP}
    K_T(s),L_T(s)\ll 
    \begin{cases} 
        T^s & \mbox{ if } s \in (n/2,n],\\
        T^{n/2}\log T & \mbox{ if } s =n/2+it.
    \end{cases}
\end{align}

In the $\SL_2(\R)$ case, we decomposed the real direction, since $\Gamma \bk \half$ was not compact in the real direction. Analogously in the current setting, we have the group decomposition $(\Gamma_1 \bk H) \overline{U} U M_1$. Since $(\Gamma_1 \bk H)$ has finite $H$-Haar measure, and $M_1$ is compact, and we have imposed a cut-off in the $\overline{U}$-direction, we are again faced with a one dimensional non-compact direction, the $U$-direction. To that end, since $\infty$ lies outside the limit set, from Lemma \ref{lem:geometric} it follows that there exists an $X$ large enough, such that 
\begin{align*}
    N_{\cP} = F_{T,X}(e)
\end{align*}
where 
\begin{align*}
    F_{T,X}(g):= \sum_{\Gamma_1\bk \Gamma} \chi_T(\gamma g) \wt{\chi}_X(\gamma g),
    \qquad \qquad \mbox{ and}\\
    \wt{\chi}_X(g):= \begin{cases}
        1 & \mbox{ if } \abs{bz^\ast(v_1\cdot g)}< X,\\
        0 & \mbox{ otherwise.}
    \end{cases}
\end{align*}
where $bz^\ast$ denotes the bend center, as in \eqref{eq:bStarZ}. Thus, note that $\wt{\chi}_X$ is also left $\Gamma_1$ invariant, and right $M_1$ invariant (since it depends only on the distance to the origin of the center, not the polar coordinate angles). 

Moreover,  a calculation shows that, in the $\vect{z}$-coordinates, $\wt{\chi}_X$ can be written as:
\begin{align*}
        \wt{\chi}_X(g):= \begin{cases}
        1 & \mbox{ if } \frac{1}{y}-X<w< \frac{1}{y}+X,\\
        0 & \mbox{ otherwise.}
    \end{cases}
\end{align*}

\textbf{The Difference Operator:} Again, we will prove an identity in terms of $K_T$ and $L_T$ for $F_{T,X}$. To that end, consider the difference operator
$$
    G_{T,X}:=F_{T,X}-K_T(\cC)F_{1,X} - L_T(\cC)F_{b,X}.
$$
By self-adjointness of $\cC$, for any $\Psi \in L^2(\Gamma \bk G /M_1)$ we have
\begin{align*}
    &G_{T,X}(\Psi):=\< G_{T,X},\Psi\>_{\G\bk G}\\
    &= 
    \int_{\G\bk G}
    \left[
    F_{T,X}(g)\Psi(g) - F_{1,X}(g)(K_T(\cC)\Psi)(g) - F_{b,X}(g)(L_T(\cC)\Psi)(g)
    \right]dg,
\end{align*}
which we can unfold to
\begin{align}\label{eq:100501 ACP}
    &G_{T,X}(\Psi)=\\ 
    \notag&\int_{\Gamma_1 \bk G} \wt{\chi}_X(g)\left(\chi_T(g)\Psi(g) - \chi_1(g)(K_T(\cC)\Psi)(g) - \chi_b(g)(L_T(\cC)\Psi)(g)\right)dg.
\end{align}
It is more convenient to work using $\vect{z}$-coordinates describing the group $G$. Now fix a fundamental domain for $\Gamma_1 \bk G /M_1$, this fundamental domain can be written in coordinates  as $\cF := P_{\Gamma_1} \times [0,\infty) \times \R \times [-\pi,\pi]^{n-1}$, where $P_{\Gamma_1}$ is a (finite $H$-volume) fundamental domain for the action of $\Gamma_1$ on $H$. Given a function $f:G\to \C$ we abuse notation and write $f(\vect{z})=f(g_{\vect{z}})$, thus if we let $\rho$ denote density of the Haar measure, we can write \eqref{eq:100501 ACP} as:
\begin{align*}
&G_{T,X}(\Psi)= \\
    &\int_{\cF}\wt{\chi}_X(y,w) \left(\chi_T(y)\Psi(\vect{z}) - \chi_1(y)(K_T(\cC)\Psi)(\vect{z}) - \chi_b(y)(L_T(\cC)\Psi)(\vect{z})\right)\rho(\vect{z})d\vect{z}.
\end{align*}
Note that by definition $\chi_T(\vect{z})$ is a function of $y$ and $\wt{\chi}_X(\vect{z})$ is a function of $w$ (and $y$), for clarity we ignore the dependence on the other variables. 

Let $\cF_X : = P_{\Gamma_1} \times [0,\infty) \times \cI_X \times[-\pi,\pi)^{n-1}$ where $\cI_X:=[1/y-X,1/y+X)$, now write
\begin{align*}
G_{T,X}(\Psi)&= 
    \int_{\cF_X} \left(\chi_T(y)\Psi(\vect{z}) - \chi_1(y)(K_T(\cC)\Psi)(\vect{z}) - \chi_b(y)(L_T(\cC)\Psi)(\vect{z})\right)\rho(\vect{z})d\vect{z}.
\end{align*}
By showing that $G_{T,X}(\Psi)=0$ for any choice of $\Psi$ we will prove the following proposition

\begin{proposition} \label{prop:main ident ACP} 
    For $\Gamma$ and $F_{T,X}$ as above we have that
    \begin{align}
        F_{T,X} = K_T(\cC) F_{1,X} + L_T(\cC)F_{b,X}
    \end{align}
    where $K_T$ and $L_T$ are the differential operators defined in \eqref{KL def ACP}. 
\end{proposition}

Once again our goal is to work on the fundamental domain of a group (rather than working with discontinuous cut-offs). Thus we will perform the same smoothing as we did in the $\SL_2(\R)$ case.

\textbf{Smoothing $\chi_T$:} Let $\sigma >0$ and let
\begin{align*}
    \chi_{1,\sigma}(\vect{z}) :=
    \begin{cases}
        1 & \mbox{ if } y < 1,\\
        0 & \mbox{ if } y > (1+\sigma),
    \end{cases}
\end{align*}
and let $\chi_{1,\sigma}$ interpolate smoothly for all values in between. Now let $\chi_{T,\sigma}(\vect{z}) : = \chi_{1,\sigma}(T y)$.

Let 
\begin{align*}
  &G_{T,X}^\sigma(\Psi) := \\
  &\int_{\cF_X} \left(\chi_{T,\sigma}(y)\Psi(\vect{z}) - \chi_{1,\sigma}(y)(K_T(\cC)\Psi)(\vect{z}) - \chi_{b,\sigma}(y)(L_T(\cC)\Psi)(\vect{z})\right)\rho(\vect{z})d\vect{z}.
\end{align*}
Now by construction 
\begin{align*}
     \int_{\cF_X} \abs{\chi_{t,\sigma}(y) - \chi_{t}(y)}^2 \rho(\vect{z})d\vect{z} \ll_{t,X} \sigma
\end{align*}
Thus by Cauchy-Schwarz we have that $\lim_{\sigma \to 0} G_{T,X}^\sigma(\Psi) = G_{T,X}(\Psi)$. Now our goal is to show that for any fixed $\vep>0$ we have $G_{T,X}^\sigma(\Psi) < \epsilon$.

\textbf{Periodizing and Smoothing $\Psi$:} Let
    \begin{align*}
        \cI_{X,\eta} : = \left[\frac{1}{y}-X+\eta, \frac{1}{y} + X -\eta\right].
    \end{align*}
    Let $\wt{\Psi}:G/M_1\to \R$ denote a function which agrees with $\Psi$ on
    \begin{align*}
        P_{\Gamma_1} \times [0, \infty)\times \cI_{X,\eta}
    \end{align*}
    for some $\eta>0$ to be chosen later. When $x_2 \not\in \cI_{X,\eta}$ we impose the condition $\wt{\Psi}(1/2y-X)=\wt{\Psi}(1/2y)=\wt{\Psi}(1/2y+X)$ for any value of the other variables, and interpolate smoothly in between.


    Using the same Cauchy-Schwarz argument as we employed in the $\SL_2(\R)$ case, we can choose $\eta$ such that the $L^2(P_{\Gamma_1}\times [0,2T]\times \cI_X \times [-\pi,\pi)^{n-1})$ cost of moving from $\Psi$ to $\wt\Psi$ is less than $\varepsilon$. Thus, Proposition \ref{prop:main ident ACP} follows if we can prove that
    \begin{align}\label{G sigma aim}
        G_{T,X,\sigma}(\wt\Psi)=0.
    \end{align}
\textbf{Working on $\cF_X$:}
    Let
    \begin{align*}
        g_{T,X}^\sigma(\vect{z}) := \wt{\chi}_X(\vect{z})(\chi_{T,\sigma}(\vect{z}) - K_T(\cC)\chi_{1,\sigma}(\vect{z}) - L_T(\cC)\chi_{b,\sigma}(\vect{z})).
    \end{align*}
Then    \eqref{G sigma aim} follows from the following lemma.
    \begin{lemma}\label{lem:G bound ACP}
      For any $\psi \in L^2(\cF_X)$ be independent of $\vect{\varphi} \in [-\pi,\pi)^{n-1}$ and any $\lambda \ge 0$ we have
      \begin{align}\label{G bound}
          \<g_{T,X}^\sigma,\psi\>_{\cF_X} \ll_{\lambda,T, \sigma, X} \|(\cC- \lambda) \psi\|_{\cF_X}  
      \end{align}
      where $\cC$ denotes the Casimir operator on $L^2(\mathcal F_X) $ (i.e the Casimir operator in $\vect{z}$-coordinates).
    \end{lemma}

    To prove \eqref{G bound}, let $\psi$ be an arbitrary function in $L^2(\cF_X)$, which is $\vect{\varphi}$ invariant. Now consider
    \begin{align*}
        \int_{\cF_X} \chi_{T,\sigma}(y) \psi(\vect{z}) \rho(\vect{z}) d \vect{z}
        =
        \int_0^{2T}\chi_{T,\sigma}(y) \int_{1/y-X}^{1/y+X}\int_{[-\pi,\pi)^{n-1}}\int_{P_{\Gamma_1}}   \psi(\vect{z}) \rho(\vect{z}) d \vect{x}d \vect{\varphi} dw dy.
    \end{align*}
    For convenience, let us change variables $w \mapsto w-1/y$, thus, the above integral becomes
    \begin{align*}
        \int_{\cF_X} \chi_{T,\sigma}(y) \psi(\vect{z}) \rho(\vect{z}) d \vect{z}
        =
        \int_0^{2T}\chi_{T,\sigma}(y) \int_{-X}^{X}\int_{[-\pi,\pi)^{n-1}}\int_{P_{\Gamma_1}}   \psi(\vect{z}) \wt{\rho}(\vect{z}) d \vect{x}d \vect{\varphi} dw dy,
    \end{align*}
    where $\wt{\rho}(\vect{z})$ is the modified Haar measure density after changing variables. Let 
    \begin{align*}
        f(y) : = \int_{X}^{-X}\int_{P_{\Gamma_1}}\int_{[-\pi,\pi)^{n-1}}   \psi(\vect{z}) \wt{\rho}(\vect{z}) d \vect{x}d \vect{\varphi} dw,
    \end{align*}
    and
    \begin{align}\label{g defn}
        h(y) : = \int_{-X}^{X}\int_{P_{\Gamma_1}}\int_{[-\pi,\pi)^{n-1}} \wt\rho(\vect{z})  (\cC-\lambda)\psi(\vect{z})  d \vect{x}d \vect{\varphi} dw.
    \end{align}

    Since, as shown in Theorem \ref{thm:rho decomp}, the density of the Haar measure decomposes into a product of densities depending on $\vect{x}$, one depending on $\vect{\varphi}$ and one depending on $w$ and $z$, we can use this fact to show
    \begin{align*}
        \int_{P_{\Gamma_1}} \cC \psi(\vect{x},y,w) \rho_H(\vect{x})d\vect{x} = \cC \int_{P_{\Gamma_1}}\psi(\vect{x},y,w) \rho_H(\vect{x}) d\vect{x}. 
    \end{align*}
    This follows from periodicity on the boundary of $\cF_X$. 
    
    Recall that the (modified) Haar measure, in $\vect{z}$ coordinates is given by $\abs{wy}^{n-1} \rho_H \rho_{M_1}$. If $n$ is odd, then the absolute value plays no role. However if $n$ is even, we need to consider the discontinuity at $0$. Because of this, we henceforth assume the harder case, when $n$ is even, the other case is similar but we have fewer boundary terms.
    
    Using the Structure Theorem for the Casimir operator (Theorem \ref{thm:Casimir})
    \begin{align}
    \begin{aligned}\label{Cas expand}
        &2\int_{\cF_X} \wt\rho(\vect{z})(\cC \psi)(\vect{z})   d\vect{z})\\
        &= 
         -\int_{0}^{X} (wy)^{n-1} \left(y^2 \partial_{y}^2 + (n+1) y \partial_{y} + 2\partial_{yw} + \frac{(n-1) (w+\frac{1}{y}) \partial_{w}}{yw}\right) \wt{\psi}(y,w)   dw \\
        &+ 
         \int_{-X}^{0} (wy)^{n-1}  \left(y^2 \partial_{y}^2 + (n+1) y \partial_{y} + 2\partial_{yw} + \frac{(n-1) (w+\frac{1}{y}) \partial_{w}}{yw}\right) \wt{\psi}(y,w)   dw.
         \end{aligned}
    \end{align}
Now we use integration by parts to handle the two derivatives in $w$
        \begin{align*}
         &\int_{0}^{X}(wy)^{n-1}  \left( 2\partial_{yw} + \frac{(n-1) (w+1/y) \partial_{w}}{yw}\right) \wt{\psi}(y,w)   dw
         \\
         &= \left[(wy)^{n-1}  \left( 2\partial_{y} + \frac{(n-1) (w+1/y)}{yw}\right) \wt{\psi}(y,w) \right]_{0}^{X} \\
         &\phantom{+}-\int_{0}^{X} (2(n-1)y(wy)^{n-2}\partial_y + (n-1)(wy)^{n-2}\\
         &\phantom{+++++}+(n-1)(n-2)(wy)^{n-3}(wy+1)
         ) \wt{\psi}(y,w)   dw 
         \end{align*}
         Thanks to our assumptions on $\wt{\psi}$, when added together, we have some cancellation coming from all the terms involving $\partial_w$. Namely 
        \begin{align*}
         &-\int_{0}^{X}(wy)^{n-1}  \left( 2\partial_{yw}+ \frac{(n-1) (w+1/y) \partial_{w}}{yw}\right) \wt{\psi}(y,w)   dw\\
         &\phantom{+++}
         +\int_{-X}^{0}(wy)^{n-1}  \left( 2\partial_{yw} + \frac{(n-1) (w+1/y) \partial_{w}}{yw}\right) \wt{\psi}(y,w)   dw\\
         &=
         \int_{0}^{X} (n-1) \left(2y(wy)^{n-2}\partial_y + (wy)^{n-2}+(n-2)(wy)^{n-3}(wy+1)
         \right) \wt{\psi}(y,w)   dw \\
         &-
         \int_{-X}^{0}(n-1) \left(2y(wy)^{n-2}\partial_y + (wy)^{n-2}+(n-2)(wy)^{n-3}(wy+1)
         \right) \wt{\psi}(y,w)   dw. 
         \end{align*}
         
         Returning to \eqref{Cas expand} we consider only the term from $[0,X)$, we thus have
         \begin{align*}
             I&:=\int_{0}^{X} (wy)^{n-1}  \left(y^2 \partial_{y}^2 + (n+1) y \partial_{y} + 2\partial_{yw} + \frac{(n-1) (w+1/y) \partial_{w}}{yw}\right) \wt{\psi}(y,w)   dw\\
             &= \int_{0}^{X} \bigg\{ (wy)^{n-1}  \left(y^2 \partial_{y}^2 + (n+1) y \partial_{y}\right)-  2(n-1)y(wy)^{n-2}\partial_y \\
             &\phantom{+++} - (n-1)(wy)^{n-2}-(n-1)(n-2)(wy)^{n-3}(wy+1)
         \bigg\} \wt{\psi}(y,w)   dw.
         \end{align*}
         Now we can use the product rule to pull the derivatives in $y$ to the front of the integral.

         Let 
         \begin{align*}
             \wt{f} := \int_0^X (wy)^{n-1}\wt{\psi}(y,w) d w.
         \end{align*}
         Then we have that 
         \begin{align*}
             &y^2 \partial_{yy} \wt{f} = w^2y^2(n-1)(n-2)(wy)^{n-3}
             + 2(w+1/y)y^2 (n-1)(wy)^{n-2}\partial_y\\
             &\phantom{++++++++++++++++}+ y^2(wy)^{n-1} \partial_{yy}.
         \end{align*}
         Rearranging and inserting into $I$ gives
        \begin{align*}
             &I=y^2\partial_{yy}\wt{f} +\int \bigg[(wy)^{n-1} (n+1) y \partial_y - 2(n-1)y(wy)^{n-2}\partial_{y}\\
            &- (n-1)(wy)^{n-2} -(n-1)(n-2)(wy+1)(wy)^{n-3}\\
            &-(wy+1)^2(n-1)(n-2)(wy)^{n-3}
             - 2y(wy+1) (n-1)(wy)^{n-2}\partial_y \bigg] \wt\psi dw.
         \end{align*}     
         Now we can collect together some of the like terms 
        \begin{align*}
          &I = y^2\partial_{yy}\wt{f}
          -\int \bigg[ (n-3)y(wy)^{n-1}\partial_{y}\\
            &\phantom{+++}
             + (n-1)(wy)^{n-1} +(n-1)(n-3)(wy+1)(wy)^{n-2}
              \bigg] \wt\psi dw.
         \end{align*}     
         Again, rearranging the product rule implies that we can write
        \begin{align*}
             I = y^2\partial_{yy}\wt{f} - (n-3)y\partial_y \wt{f} - (n-1)\wt{f}.
         \end{align*} 
         
         Putting together the two integrals, we conclude that 
         \begin{align}\label{f differential equation}
             -y^2\partial_{yy}f + (n-3)y\partial_y f + (n-1) f  = h(y).
         \end{align}
         If we write $\lambda=s(n-s)$, then 
          \begin{align*}
             &-y^2\partial_{yy}f + (n-3)y\partial_y f + (n-1)f=\lambda f.
         \end{align*} 
         has solutions
         \begin{align*}
            f(y) = \frac{y^{s}}{y}, 
            \qquad\mbox{ and } \qquad
            f(y) = \frac{y^{n-s}}{y}
         \end{align*}
         As a consequence of \eqref{f differential equation}, using the method of variation of parameters (we omit the proof which is the same as the $\SL_2(\R)$-case (see \eqref{diff eq}), we have the following corollary
\begin{theorem}\label{thm:f explicit}
  For any $s>0$, with $s \neq n/2$, there exist constants $c_1$ and $c_2$ such that:
  \begin{align}\label{f explicit}
      f(y) = c_1\frac{y^{s}}{y} + c_2 \frac{y^{n-s}}{y} &+ \frac{y^{s}}{y}u(y) + \frac{y^{n-s}}{y} v(y)
  \end{align}
  where
  \begin{align*}
    &u(y):= \int_y^T \frac{t^{2-s}}{2s-n}h(t) dt,
    \qquad\qquad 
    v(y):= \int_y^T \frac{t^{s-n+2}}{n-2s}h(t) dt,
  \end{align*}
  If $ s = n/2$, then 
  \begin{align}
    f(y) = c_1\frac{y^{n/2}}{y} + c_2\frac{y^{n/2}}{y} \log y + u(y) + v(y)\log y,
  \end{align}
  where the inhomogeneous terms are given by
  \begin{align*}
    &u(y):= \int_y^T \frac{-t^{2-n/2}\log(t)h(t)}{(n-2)\log(t)+1} dt,
    \qquad
    v(y):= \int_y^T \frac{t^{2-n/2}h(t)}{(n-2)\log(t)+1} dt.
  \end{align*}
\end{theorem}
    
With Theorem \ref{thm:f explicit} at hand, we may prove Lemma \ref{lem:G bound ACP}.

\begin{proof}[Proof of Lemma \ref{lem:G bound ACP}]
  We will prove the bound for the term depending on $u$ as the same proof applies for the term depending on $v$. Thus, consider
  \begin{align*}
    \int_0^T \frac{y^{s}}{y}u(y) dy = \int_0^T y^{s-1} \int_y^T \frac{t^{2-s}}{2s-n}h(t) dt dy.
  \end{align*}
  Integrating by parts, and applying the fundamental theorem of calculus then gives
  \begin{align*}
     \int_0^T \frac{y^{s}}{y}u(y) dy  = c\left[ y^{s} u(y)\right]_{y=0}^T - 
    c'\int_0^T y^2 h(y) dy.
  \end{align*}
  Now note that $u(T)=0$, thus:
  \begin{align*}
     \int_0^T \frac{y^{s}}{y}u(y) dy  = \left(\lim_{y \to 0} y^{s} \int_y^T \frac{t^{2-s}}{2s-n}h(t) dt\right) - 
    c'\int_0^T y^2h(y) dy.
  \end{align*}
  Recall $h(y):=\int_{-X}^{X}\int_{P_{\Gamma_1}}\int_{[-\pi,\pi)^{n-1}} \rho(\vect{z})  (\cC-\lambda)\psi(\vect{z})  d \vect{x}d \vect{\varphi} dw$, thus the first term behaves like $\lim_{y \to 0} y^3=0$. Thus
  \begin{align*}
    \int_0^T \frac{y^{s}}{y}u(y) dy\ll
    \int_0^T \int_{-X}^{X}\int_{P_{\Gamma_1}}\int_{[-\pi,\pi)^{n-1}} y^2\rho(\vect{z})  (\cC-\lambda)\psi(\vect{z})  d \vect{x}d \vect{\varphi} dw dy.
  \end{align*}
  Applying Cauchy-Schwarz:
  \begin{align*}
    \int_0^T \frac{y^{s}}{y}u(y) dy &\ll
    \left(\int_0^T \int_{-X}^{X}\int_{P_{\Gamma_1}}\int_{[-\pi,\pi)^{n-1}} \rho(\vect{z})  \abs{(\cC-\lambda)\psi(\vect{z})}^2  d \vect{x}d \vect{\varphi} dw dy\right)^{1/2}\\
    &\phantom{\ll++++}\cdot \left(\int_0^T \int_{-X}^{X}\int_{P_{\Gamma_1}}\int_{[-\pi,\pi)^{n-1}} y^4 \rho(\vect{z})  d \vect{x}d \vect{\varphi} dw dy\right)^{1/2}\\
    &\ll_{T,X,\lambda} \|(\cC-\lambda)\psi(\vect{z}) \|_{\cF_X} .
  \end{align*}
  Note that, here, it is crucial that we imposed the cut-off in the $w$-direction. This ensures that the second factor on the right hand side is finite. The proof for the $v$-term is identical.
\end{proof}    

Thus, working our way back up, with the same argument as in the $\SL_2(\R)$ setting,
we conclude the proof of Proposition \ref{prop:main ident ACP}.

\subsection{Proof of Theorem \ref{thm:w count}}
    
    As for the $SL_2(\R)$ case, we now return to the smooth count with $U$-cutoff:
    \begin{align*}
        \wt{N}_{\Gamma}(T) = \<F_{T,X}, \Psi\>_\Gamma.
    \end{align*}
    Now apply the abstract Parseval's identity \eqref{API}
    \begin{align}
        \<F_{T,X},\Psi\>_\Gamma 
        &= \<\wh{F_{T,X}},\wh{\Psi}\>_{\Spec(\Gamma)} \notag\\ 
        &= \wh{F_{T,X}}(\lambda_0)\wh{\Psi}(\lambda_0) + \int_{\Spec(\Gamma) \setminus \{\lambda_0\}} \wh{F_{T,X}}(\lambda)\wh{\Psi}(\lambda) \mathrm{d}\nu(\lambda).
    \label{spectral ip ACP}
    \end{align} 
    As with the $\SL_2(\R)$ case, we can apply spectral methods to extract the $T$ and $\vep$-dependence. Applying the abstract spectral theorem gives 
    \begin{align*}
        \wh{F_{T,X}}(\lambda_0)\wh{\Psi}(\lambda_0) = \<\operatorname{Proj}_{\cH_0}(F_{T,X}),\operatorname{Proj}_{\cH_0}(\Psi)\>.
    \end{align*}
    Then by linearity and our main identity Proposition \ref{prop:main ident ACP} we conclude
    \begin{align*}
      \wh{F_{T,X}}(\lambda_0)\wh{\Psi_\varepsilon}(\lambda_0) &= K_T(\lambda_0) \<\operatorname{Proj}_{\cH_0}(F_{1,X}),\operatorname{Proj}_{\cH_0}(\Psi_\varepsilon)\>\\
      &\phantom{=} + L_T(\lambda_0) \<\operatorname{Proj}_{\cH_0}(F_{b,X}),\operatorname{Proj}_{\cH_0}(\Psi_\varepsilon)\>\\
      &= T^\delta  \<H,\operatorname{Proj}_{\cH_0}(\Psi_\varepsilon)\> + O(T^{n/2}),
    \end{align*}
    where $H: = c_1\operatorname{Proj}_{\cH_0}(F_{1,X})+c_b\operatorname{Proj}_{\cH_0}(F_{b,X})$ for some constants $c_1,c_b$.

The problem remains to determine the $\vep$ dependence of the projection operator $\operatorname{Proj}_{\cH_0}(\Psi_\vep)$. In general, this projection can be realized in a number of ways, either as a Burger-Roblin-type measure of $\Psi_\vep$ (see \cite[p. 861]{MohammadiOh2015}), or using representation-theoretic decompositions as in \cite{BourgainKontorovichSarnak2010, Vinogradov2012}. We will give a soft argument that avoids either.

    We know from 
     \eqref{Count eps bounds} that
    \begin{align*}
      N^\cP\left(\frac{T}{1+c\vep}\right) \le \wt{N}_{\Gamma,\vep}(T) \le N^{\cP}\left(\frac{T}{1-c\vep}\right),
    \end{align*}
    for any value of $\vep$ and any value of $T$. However we also know \emph{a priori} (e.g., using \cite{Kim2015}) that
    \begin{align*}
      N^\cP\left(\frac{T}{1\pm c\vep}\right) = c_\cP\left(\frac{T}{1\pm c\vep}\right)^\delta(1 + o(1)),
    \end{align*}
    as $T\to\infty$.
    Dividing by $T^\gd$ then gives 
    \begin{align*}
      c_\cP\left(\frac{1}{1+c\vep}\right)^\delta - o(1)\le
        \<H,\operatorname{Proj}_{\cH_0}(\Psi_\varepsilon)\> 
      \le
      c_\cP\left(\frac{1}{1-c\vep}\right)^\delta + o(1).
    \end{align*}
    Now send $T\to \infty$ and Taylor expand $\frac{1}{(1\pm c\vep)^\delta}$ in $\vep$, giving:
    \begin{align*}
      \<H,\operatorname{Proj}_{\cH_0}(\Psi_\varepsilon)\> = C + O(\vep).
    \end{align*}
    Hence
    \begin{align}
      \wh{F_{T,X}}(\lambda_0)\wh{\Psi_\varepsilon}(\lambda_0) 
      &= T^\delta c(1+O(\varepsilon)) + O(T^{n/2})
    \end{align}
    for some constant $c$ independent of $\varepsilon$. (Note that this positivity argument does not apply to the other eigenvalues. Hence with sharp cutoffs, as in Theorem \ref{thm:main}, we do not extract lower order terms.)

    Turning now to the remainder, after extracting the main term corresponding to $\lambda_0$ we are left with
    \begin{align*}
        \Err &:= \int_{\Spec(\Gamma) \setminus \{\lambda_0\}}\wh{F_{T,X}}(\lambda) \wh{\Psi}(\lambda) d\nu\\
        &=  \int_{\Spec(\Gamma) \setminus \{\lambda_0\}}\left(K_T(\lambda) \wh{F_{1,X}}(\lambda) + L_T(\lambda) \wh{F_{b,X}}(\lambda)  \right) \wh{\Psi}(\lambda) d\nu.
    \end{align*}
Assume for simplicity that there are no other discrete eigenvalues above the base.
     Now apply the abstract spectral theorem and the bounds from \eqref{K and L bounds ACP} to conclude that 
     \begin{align*}
         \int_{\Spec(\Gamma) \setminus \{\lambda_0\}}K_T(\lambda) \wh{F_{1,X}}(\lambda) \wh{\Psi}(\lambda) d\nu
         &\ll
         T^{n/2} \log T 
         \int_{\Spec(\Gamma)\setminus\{\lambda_0\}} \wh{F_{1,X}}(\lambda) \wh{\Psi}(\lambda) d\nu .
     \end{align*}
     Now apply Cauchy-Schwarz and Parseval to get
     \begin{align*}
         &\ll \ 
         T^{n/2} \log T
         \left(\int_{\Spec(\Gamma)\setminus\{\lambda_0\}} \wh{F_{1,X}}(\lambda)^2 d\nu\right)^{1/2}\left(\int_{\Spec(\Gamma)\setminus\{\lambda_0\}} \wh{\Psi}(\lambda)^2 d\nu\right)^{1/2}\\
                  &\le \ 
         T^{n/2} \log T
         \left(\int_{\Spec(\Gamma)} \wh{F_{1,X}}(\lambda)^2 d\nu\right)^{1/2}\left(\int_{\Spec(\Gamma)} \wh{\Psi}(\lambda)^2 d\nu\right)^{1/2}\\
                  &= \ 
         T^{n/2} \log T
         \ \|F_{1,X}\|_\Gamma \  \|\Psi\|_{\Gamma}.
     \end{align*}
    Finally, note that since $\psi_\varepsilon$ is normalized to have unit $L^1$-mass, we have that $\|\Psi\|_\Gamma \ll \varepsilon^{-(n+1)/2}$.
    In the case of other eigenvalues, we replace the bound $T^{n/2}\log T$ above with $T^{s_1}$.
    This completes the proof of Theorem \ref{thm:w count}.    \qed
    
    \begin{remark}\label{rmk:Y}
        If we remove our assumption that $L^2(\Gamma\bk G)$ does not weakly contain any nonspherical complementary series representations, then, after removing contributions from Laplace eigenvalues in \eqref{spectral ip ACP}, 
        the remainder would not necessarily be tempered. So instead of getting an error of order $T^{n/2}\log T$, we would only be able to bound what remains by $O(T^{n-1})$, corresponding to the spectral parameter of any potential nonspherical complimentary series.
    \end{remark}

  \small 
  \bibliographystyle{alpha}
  \bibliography{biblio}

\end{document}